\documentclass[11pt,a4paper]{article}

\usepackage[hmargin={25mm,25mm},vmargin={25mm,25mm}]{geometry}

\usepackage{verbatim, ifpdf}
\usepackage{graphicx, color}
\usepackage{multirow}
\usepackage{amsthm}
\usepackage{amssymb, amsmath}
\usepackage{caption}
\usepackage{subcaption}
\usepackage{afterpage}
\usepackage{enumerate}
\usepackage{enumitem}
\usepackage{empheq}
\usepackage{scalerel}
\usepackage{multicol}
\usepackage{yhmath} 
\usepackage{stmaryrd}
\usepackage{anyfontsize}
\usepackage{wrapfig}
\usepackage{verbatim}
\usepackage{url}

\usepackage{bbm}
\usepackage{tikz}
\usetikzlibrary{calc}
\usetikzlibrary {decorations}
\usepackage{pgfplots}
\pgfplotsset{compat=1.14}
\usepackage{underscore}
\usepackage[hypertexnames=false]{hyperref}
\usepackage[capitalize]{cleveref}

\newcommand{\vc}[1]{\boldsymbol{#1}}

\newcommand{\ul}[1]{\underline{#1}}

\newcommand{\RR}{\mathbb{R}}

\newcommand{\PP}{\mathbb{P}}
\newcommand{\EE}{\mathbb{E}}

\newcommand{\eps}{\varepsilon}
\newcommand{\cat}[1]{\mathcal{#1}}

\newcommand{\mono}{\hookrightarrow}
\newcommand{\norm}[1]{\left\lVert#1\right\rVert}
\newcommand{\inner}[1]{\left\langle#1\right\rangle}
\newcommand{\abs}[1]{\left|#1\right|}

\newcommand{\numberthis}{\addtocounter{equation}{1}\tag{\theequation}}

\theoremstyle{definition}
\newtheorem*{theorem*}{Theorem}
\newtheorem{theorem}{Theorem}[section]
\newtheorem{lemma}[theorem]{Lemma}
\newtheorem{corollary}[theorem]{Corollary}
\newtheorem{conjecture}[theorem]{Conjecture}
\newtheorem{question}[theorem]{Question}
\newtheorem{definition}[theorem]{Definition}

\newtheorem{example}[theorem]{Example}
\newtheorem{notation}[theorem]{Notation}

\newtheorem*{remark}{Remark}

\theoremstyle{plain}
\newtheorem*{claim*}{Claim}

\begin{document}
 \title{On the maximum degree of induced subgraphs of the Kneser graph}
 \author{Hou Tin Chau\footnote{School of Mathematics, University of Bristol, UK.\ Supported by an EPSRC\ Doctoral Studentship.},\, David Ellis\footnote{School of Mathematics, University of Bristol, UK},\, Ehud Friedgut\footnote{Department of Mathematics, Weizmann Institute of Science, Israel.}\,\, and\ Noam Lifshitz\footnote{Einstein Institute of Mathematics, Hebrew University of Jerusalem, Israel. Supported in part by ISF\ grant 1980/22.}}
\date{November 2024}
\maketitle

\begin{abstract}
For integers $n \geq k \geq 1$, the {\em Kneser graph} $K(n, k)$ is the graph with vertex-set consisting of all the $k$-element subsets of $\{1,2,\ldots,n\}$, where two $k$-element sets are adjacent in $K(n,k)$ if they are disjoint. We show that if $(n,k,s) \in \mathbb{N}^3$ with $n > 10000 k s^5$ and $\cat F$ is set of vertices of $K(n,k)$ of size larger than $\{A \subset \{1,2,\ldots,n\}:\ |A|=k,\ A \cap \{1,2,\ldots,s\} \neq \varnothing\}$, then the subgraph of $K(n,k)$ induced by $\cat F$ has maximum degree at least \[ \left(1 - O\left(\sqrt{s^3 k/n}\right)\right)\frac{s}{s+1} \cdot {n-k \choose k} \cdot \frac{\abs{\cat F}}{\binom{n}{k}}.\]
This is sharp up to the behaviour of the error term  $O(\sqrt{s^3 k/n})$. In particular, if the triple of integers $(n, k, s)$ satisfies the condition above, then the minimum maximum degree does not increase `continuously' with $\abs{\cat F}$. Instead, it has $s$ jumps, one at each time when $\abs{\cat F}$ becomes just larger than the union of $i$ stars, for $i = 1, 2, \ldots, s$. An appealing special case of the above result is that if $\mathcal{F}$ is a family of $k$-element subsets of $\{1,2,\ldots,n\}$ with $|\mathcal{F}| = {n-1 \choose k-1}+1$, then there exists $A \in \mathcal{F}$ such that $\mathcal{F}$ is disjoint from at least
$$\left(1/2-O\left(\sqrt{k/n}\right)\right){n-k-1 \choose k-1}$$
of the other sets in $\mathcal{F}$; we give both a random and 
a deterministic construction showing that this is asymptotically sharp if $k=o(n)$. In addition, it solves (up to a constant multiplicative factor) a problem of Gerbner, Lemons, Palmer, Patk\'os and Sz\'ecsi \cite{almostintersecting}.

Frankl and Kupavskii \cite{maxdegFK}, using different methods, have recently proven similar results under the hypothesis that $n$ is at least a quadratic in $k$.
\end{abstract} 
\pagebreak

\section{Introduction}

For $n \in \mathbb{N}$, we write $[n]: = \{1,2,\ldots,n\}$ for the standard $n$-element set, and for a set $X$, we write ${X \choose k}: = \{A \subset X:\ |A|=k\}$. The {\em Kneser graph} $K(n,k)$ is the graph $(V,E)$ with $V = {[n] \choose k}$ and $E = \{\{A,B\}:\ A \cap B = \varnothing\}$, i.e.\ two $k$-sets are joined by an edge of the Kneser graph iff they are disjoint. The Kneser graph $K(n,k)$ has independent sets of size ${n-1 \choose k-1}$, viz., the sets of the form $\left\{A \in {[n] \choose k}:\ i \in A\right\}$ for $i \in [n]$, and the well-known Erd\H{o}s-Ko-Rado theorem states that no independent set can have larger size. (Recall that a set $I$ of vertices in a graph $G$ is said to be {\em independent} if there is no edge of $G$ between any two of the vertices in $I$.)

Our work in this paper starts with the following question: if $\mathcal{F}$ is a set of vertices of $K(n,k)$ with size ${n-1 \choose k-1}+1$, how large must the maximum degree of the induced subgraph $K(n,k)[\mathcal{F}]$ be? (The Erd\H{o}s-Ko-Rado theorem implies that this maximum degree must be at least one.) Put another way, if $\mathcal{F}$ is a family of ${n-1 \choose k-1}+1$ $k$-element subsets of $[n]$, must there exist a set in $\mathcal{F}$ which is disjoint from many of the other sets in $\mathcal{F}$?

This question in turn was motivated by the beautiful (and very short) resolution of the sensitivity conjecture, by Huang \cite{sensitivity}. Huang proved that for any $n \in \mathbb{N}$, the subgraph of the hypercube $Q_n$ induced by a set of $2^{n-1}+1$ vertices must have maximum degree at least $\sqrt{n}$. (This is best-possible whenever $n$ is a perfect square, and it implies the sensitivity conjecture, as had previously been observed by Gotsman and Linial.) This result is striking, because $Q_n$ has independent sets of size $2^{n-1}$, but going just one above this size, forces the maximum degree of the induced subgraph to jump from $0$ to $\sqrt{n} = \sqrt{\Delta(Q_n)}$. Huang's proof is by the construction of a signed adjacency matrix for $Q_n$ which has appropriate eigenvalues, together with an application of Cauchy's interlacing theorem.

The following questions suggest themselves: is there a similar `sharp jump' for the Kneser graph, and if so, is there a proof of this along the lines of Huang's proof?

In this paper, we prove that there is indeed such a sharp jump for the Kneser graph (for $k \leq cn$ where $c>0$ is an absolute constant); though we were unable to find a proof along the same lines as Huang's (and our proof is a good deal longer). We prove the following.

\begin{theorem}
\label{thm:main}
Let $(n,k,s) \in \mathbb{N}^3$ with $n \geq 10000ks^5$. Let $\mathcal{F} \subset {[n] \choose k}$ such that $|\mathcal{F}| \geq {n \choose k}-{n-s \choose k}$, and suppose that $\mathcal{F}\neq \left\{A \in {[n] \choose k}:\ A \cap S \neq \varnothing\right\}$, for all $S \in {[n] \choose s}$. Then the subgraph of $K(n,k)$ induced by $\mathcal{F}$ has maximum degree at least
\[ \left(1 - O\left(\sqrt{s^3 k/n}\right)\right)\frac{s}{s+1} \cdot {n-k \choose k} \cdot \frac{\abs{\cat F}}{\binom{n}{k}},\]
i.e.\ there exists $A \in \mathcal{F}$ such that $A$ is disjoint from at least
\[ \left(1 - O\left(\sqrt{s^3 k/n}\right)\right)\frac{s}{s+1} \cdot {n-k \choose k} \cdot \frac{\abs{\cat F}}{\binom{n}{k}}\]
of the sets in $\mathcal{F}$.
\end{theorem}
Our proof combines spectral techniques with more combinatorial `stability' arguments. Observe that the $s=1$ case of the above theorem has the following appealing corollary, which is a rough analogue of Huang's result on $Q_n$.

\begin{corollary}
\label{cor:main}
For each $\epsilon >0$, there exists $\delta >0$ such that the following holds. Suppose $(n,k) \in \mathbb{N}^2$ with $k \leq \delta n$. Let $\mathcal{F} \subset {[n] \choose k}$ with $|\mathcal{F}| = {n-1 \choose k-1}+1$. Then there exists $A \in \mathcal{F}$ such that $A$ is disjoint from at least $(1/2-\epsilon){n-k-1 \choose k-1}$ of the sets in $\mathcal{F}$.
\end{corollary}
We note that, for $k = \Theta(n)$, this corollary says that when $|\mathcal{F}|$ increases from ${n-1 \choose k-1}$ to ${n-1 \choose k-1}+1$, the minimal maximum degree of the subgraph $K(n,k)[\mathcal{F}]$ jumps from zero to within a constant factor of $\Delta(K(n,k))$, which is (in a sense) an even more extreme jump than in Huang's theorem. (In the latter, the minimal maximum degree jumps from zero to $\sqrt{\Delta(Q_n)}$.)

The following construction shows that, up to the dependence of $\delta$ upon $\epsilon$, the above corollary is sharp. Take
$$\mathcal{F}=\left\{A \in {[n] \choose k}:\ \{1,2\} \subset A\right\} \cup \{A \cup \{1\}: A \in \mathcal{A}_1\} \cup \{A \cup \{2\}:\ A \in \mathcal{A}_2\},$$
where $\mathcal{A}_1$ and $\mathcal{A}_2$ are initial segments of the colex ordering on ${[n] \setminus \{1,2\} \choose k-1}$ with sizes as equal as possible. In the next section we give a random construction that achieves (asymptotically) the same bound.

We note that Frankl and Kupavskii \cite{maxdegFK}, using different methods, recently proved analogues of the above results under the (stronger) hypothesis that $k \leq c\sqrt{n}$ (for an absolute constant $c>0$).

A closely related problem (which has been more studied to date than our `maximum degree' problem), is to minimize the number of edges in subgraph of the Kneser graph induced by a family $\mathcal{F}\subset {[n] \choose k}$, over all families $\mathcal{F}$ of fixed size.

\begin{question}
        For each $n \geq k \geq 1$ and each $1 < m < {n \choose k}$, what is the minimum number of edges $e(K(n, k)[\cat F])$ over all $\mathcal{F} \subset {[n] \choose k}$ with $|\mathcal{F}|=m$? That is, if $|\mathcal{F}|=m$, how many (unordered) pairs $\{A, B\}$ can we guarantee to find such that $A,B \in \cat F$ and $A \cap B = \varnothing$?
\end{question}

Das, Gan and Sudakov \cite{edgenum} showed that, for sufficiently large $n$ and provided $m$ is not too large, to minimize the above quantity over all families of size $m$, it is best to take the $k$-subsets one by one according to the lexicographical order. (Recall that the {\em lexicographical order} or {\em lex order} $\prec$ on $\binom{[n]}{k}$ is defined by $A \prec B$ if $\min(A \Delta B) \in A$.)

\begin{theorem}[\cite{edgenum} Theorem 1.6]\label{edgenumthm}
    Let $A_1 \prec A_2 \prec \cdots \prec A_{\binom{n}{k}}$ be the lex order on $\binom{[n]}{k}$.
    If $n > 108 k^2 s (k+s)$ and $m \le \binom{n}{k} - \binom{n-s}{k} $, 
    then the initial segment $\{A_1, A_2, \ldots, A_m\}$ is a family that minimizes  $e(K(n, k)[\cat F])$ subject to $\abs{\cat F} = m$.
\end{theorem}

The lower bound in Theorem \ref{thm:main}, combined with a general construction we will shortly describe (see (\ref{upperbound}) in \cref{explicitconstruction}), shows that for $n \ge 10000 ks^5$, a plot of $\min \Delta (\cat F)$ against $\abs{\cat F}$ looks roughly as in \cref{roughplot}.

We note that the $n = \Theta (k)$ regime is particularly interesting because the Kneser graph has edge-density exponentially small in $n$, in this regime, making it `more surprising' in a sense that the maximum degree of an induced subgraph of size greater than a star, is as large as we demonstrate. Indeed, the edge-density of the Kneser graph is roughly $\left.\binom{n-k}{k} \middle/\binom{n}{k}\right. \approx \left(1-\frac{k}{n}\right)^k$, so in the $n = \Omega(k^2)$ regime, the density is bounded away from zero by an absolute positive constant, whereas for $n = \Theta(k)$ the density is exponentially small in $n$, i.e., two uniformly random subsets of $[n]$ of size $k$ will intersect with `very high probability'. 

Our results also solve (up to a constant multiplicative factor) a problem of Gerbner, Lemons, Palmer, Patk\'os, and Sz\'ecsi, who in \cite{almostintersecting} asked the following.
\begin{question}[\cite{almostintersecting} Section 4, `seems to be an interesting problem']
\label{almostinersectingquestion}
Let $k,l \in \mathbb{N}$. We say $\cat F \subseteq \binom{[n]}{k}$ is $(\le l)$-\emph{almost intersecting} if $\Delta(K(n, k)[\cat F]) \le l$. What is the smallest $n_0 = n_0(k,l) \in \mathbb{N}$ such that for all $n \ge n_0$, the largest $(\le l)$-almost intersecting families are just the stars $\cat D_i$ of size $\binom{n-1}{k-1}$?
\end{question}

Gerbner, Lemons, Palmer, Patk\'os, and Sz\'ecsi \cite{almostintersecting} showed that $n_0 = \min\{ O(k^3 + kl),  O(k^2 l)\}$
 is enough. Corollary \ref{cor:main} implies that for $n \leq ck $ (for $c>0$ a sufficiently small absolute constant), we have 
 \[\abs{\cat F} \ge \binom{n-1}{k-1} + 1 \implies \Delta(K(n, k)[\cat F]) \ge 0.49 \binom{n-k-1}{k-1}.\] 
 In particular the right-hand side is greater than $\left(\frac{n}{2k}\right)^{k-1}$, so $n = \Omega\left(k l^{1/(k-1)}\right) $ is enough to guarantee that the maximum degree is larger than $l$, so any family of size $\binom{n-1}{k-1}+1$ cannot be $(\le l)$-almost intersecting. 
 The same method works for families of size $\binom{n-1}{k-1}$ that are not stars, as we shall see in \cref{equalitycase}.
 This shows that the minimum $n_0$ in  \cref{almostinersectingquestion} satisfies $n_0 = \Theta \left(k l^{1/(k-1)}\right)$, since if $n = \lfloor k l^{1/(k-1)} / 100 \rfloor$, any family $\cat F \subset \binom{n}{k}$ with $\abs{\cat F} = \binom{n-1}{k-1} + 1$ trivially has 
 \[\Delta(\cat F) \le \binom{n-1}{k-1} \le \frac{n^{k-1}}{(k-1)!} \le \frac{n^{k-1}}{(k/10)^{k-1}} < l. \]
 \newline

Before discussing our strategy for proving Theorem \ref{thm:main}, it is convenient to introduce the following notation. A {\em star} $\cat D_i$
 is a subset of ${[n] \choose k}$ of the form $\left\{A \in {[n] \choose k}:\ i \in A\right\}$, for some $i \in [n]$. The families $\left\{A \in {[n] \choose k}:\ A \cap S \neq \varnothing\right\}$, for $S \in {[n] \choose s}$, appearing in the statement of Theorem \ref{thm:main}, are precisely unions of $s$ distinct stars.

 Note that a union of $s$ stars has size $\binom{n}{k} - \binom{n-s}{k}$. We interpolate linearly between integer values of $s$, using a `size parameter' $\lambda$ to describe the size of $\cat{F}$, so that if $\mathcal{F}$ has size parameter $\lambda$, then its size is the same as that of a union of $\lambda$ stars, when $\lambda \in \mathbb{N}$. 
 
 \begin{notation}[Size parameter]
For $s \le \lambda \le s+1$, we say $\cat F$ {\em has size parameter $\lambda$} if 
\begin{align*}
    \abs{\cat F} &= \binom{n}{k} - \binom{n - s }{k} + (\lambda - s ) \binom{n - s - 1}{k-1} \\ 
    &=  (\lambda -s) \binom{n - s - 1}{k-1} + \sum_{i=1}^{s} \binom{s}{i}\binom{n - s}{k-i}. \numberthis\label{Fsize}
\end{align*}
Equivalently,
\begin{align} \abs{\cat F} 
&= \binom{n}{k} - \binom{n - s - 1}{k} + (\lambda - s - 1) \binom{n - s - 1}{k-1} \nonumber \\
&= \lambda \binom{n - s - 1}{k-1} + \sum_{i=2}^{s+1}\binom{s+1}{i} \binom{n - s- 1}{k-i}.  \nonumber
\end{align}
where $s \le \lambda \le s+1$.
\end{notation}
We note, for later, that if $\cat F$ has size parameter $\lambda$, then 
\begin{equation}
\label{crudesize}
\lambda \binom{n - 1}{k - 1} - \binom{s+1}{2} \binom{n - 2}{k - 2} \le \abs{\cat F} \le \lambda \binom{n - 1}{k - 1}. 
\end{equation}
(This can be verified just by checking the cases $\lambda = s$ and $\lambda = s+1$, since all the quantities above are affine linear functions of $\lambda$.)

\begin{figure}[h]\centering
\begin{tikzpicture}[scale = 2, decoration={zigzag, segment length=2mm}]
\draw[blue, line width = 1mm] (0, 0) --(1, 0);
\filldraw[blue] (1, 0) circle (1pt);

\draw[help lines, color=gray!30, dashed] (-0.1, -0.1) grid (5.2,4.2);
\draw[->, thick] (0,0)--(5.2,0) node[right]{$\abs{\cat F}/\binom{n-1}{k-1}$};
\draw[->, thick] (0,0)--(0,4.2) node[above]{$\min \Delta(\cat F) / \binom{n-k-1}{k-1}$};

  \foreach \a in {1,2,...,4} {
  \pgfmathparse{\a * \a / (\a + 1)}
    \coordinate (begin\a) at  (\a ,  \pgfmathresult) { };
    \coordinate (end\a) at (\a + 1, \a) { };
    \draw[blue, ultra thick, decorate] (begin\a) -- (end\a) ;
    \filldraw[blue] (end\a) circle (1 pt);
    \draw[blue] (begin\a) circle [radius = 3pt]; 

      \draw[smooth, domain={\a:\a + 1}, red, variable=\x] plot ({\x}, {(\a * (\a - 1) / 2 + (\x - \a) * \a) * 2 /\x});

  }

\node[label={[label distance=0.3cm]below:$1$}] at (1, 0) {};
 \foreach \a in {2,3,4,5} {
    \node[label={[label distance=0.3cm]below:$\approx \a$}] at (\a, 0) {};
  }

  \node[label={[label distance=0.3cm]left:$1$}] at (0, 1) {};
 \foreach \a in {2,3,4} {
    \node[label={[label distance=0.3cm]left:$\approx \a$}] at (0, \a) {};
  }
\end{tikzpicture}
\caption{Rough sketch of $\min \Delta (\cat F)$ against $\abs{\cat F}$ when $n$ is sufficiently large. Zigzag and hollow circles are schematic for a gap (not up to scale and depends on the ratio $p = k/n$) between the lower bound \cref{lowerbound} and the upper bound (\ref{upperbound}). The thin red curve is the average degree of the initial segment of lex ordering, corresponding to the lower bound one gets from applying Theorem \ref{edgenumthm}.}
\label{roughplot}
\end{figure}
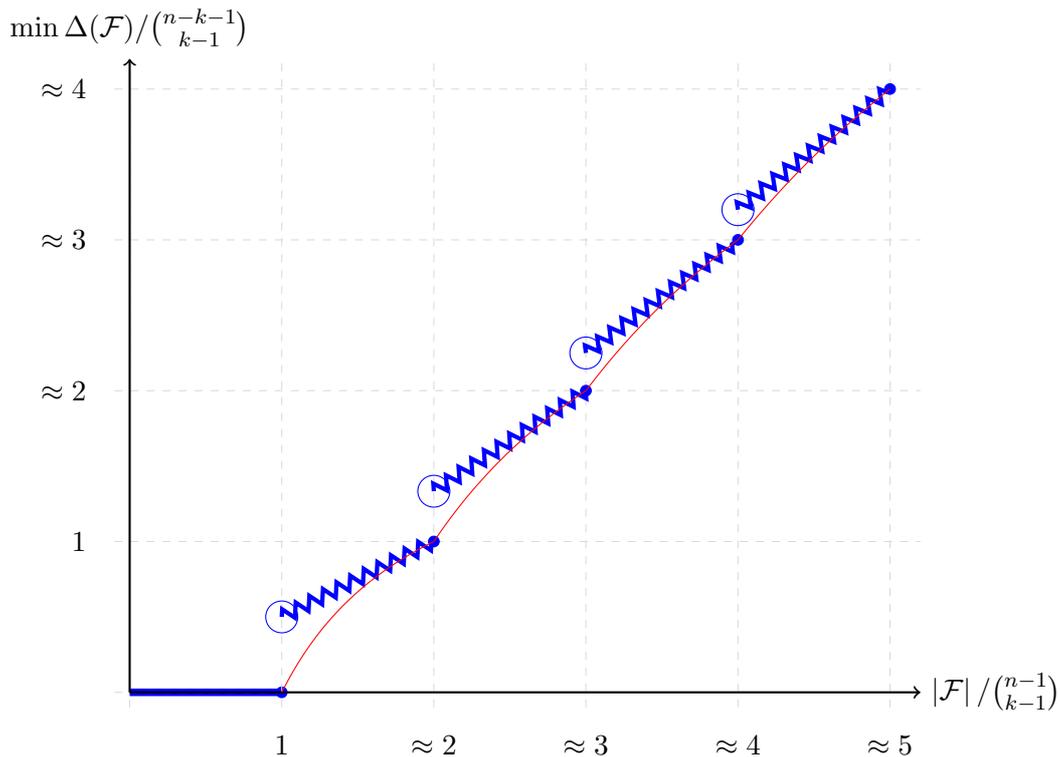

\subsection*{Proof strategy}
The key idea in the proof of Theorem \ref{thm:main} is to show that, if $\mathcal{F}$ has size parameter $\lambda$ with $s \leq \lambda \leq s+1$ and the maximum degree of $\cat F$ is `small', then in fact we can find $s+1$ stars that cover most of $\cat F$ (\cref{manylem3}). 

Suppose $\mathcal{F}$ has size parameter $\lambda$ with $s \leq \lambda \leq s+1$. Since the size of $\cat F$ is very roughly $\lambda$ times that of a star, 
this will require finding stars in which $\cat F$ has density roughly $\lambda/(s+1)$. There are $n$ stars, and each $k$-set is contained in exactly $k$ of them, so the obvious averaging argument only gives a star where $\cat F$ has density $\lambda k/n$, which is not enough. 

Instead, we use the following strategy. Consider the indicator function $f = \mathbbm 1_{\cat F}:\binom{[n]}{k} \to \{0, 1\}$. This can also be viewed as a vector in $\RR^{\binom{[n]}{k}}$. We can decompose it into a sum of eigenvectors of the (adjacency matrix of the) Kneser graph (these eigenvectors are given in \cref{singvalkneser}), say as 
\[f = \alpha \mathbbm 1_{\binom{[n]}{k}} + f_1 + \cdots + f_k\]
where $\alpha = \EE(f)$ and $f_i$ is an eigenvector of the Kneser graph with corresponding eigenvalue equal to $(-1)^i \binom{n-k-i}{k-i}$.

We show that, if $f_1$ has small $L^2$-norm, then the number of edges of $K(n,k)$ within $\mathcal{F}$ is close to that in a random subset (this is the content of \cref{thirdorder}), so the average degree is large. Otherwise, $f_1$ has large $L^2$-norm (larger than constant times $\norm{f}_2$), and this must be due to $\mathcal{F}$ having constant density inside some star (see \cref{lincompduetostar}). And whenever $\mathcal{F}$ has constant density inside some star, the Expander Mixing Lemma can be used to show that there are many edges of $K(n,k)[\mathcal{F}]$ between this star and its complement. This in turn enables us to show that the density of $\mathcal{F}$ inside the relevant star is close to $1/2$ (in the setting of Corollary \ref{cor:main}), or close to $\lambda/(s+1)$ (in the general case); this is the content of \cref{extlem3}. And this is enough to conclude the proof.


\section{Upper bounds, and (more) general conjectures}
In this section, we describe two constructions of induced subgraphs of $K(n, k)$ with small maximum degree, for any given size parameter $\lambda \in [s, s+1]$ such that the quantity in (\ref{Fsize}) is an integer. The first is a random construction, and the second is explicit. The upper bound given by the second construction will be used later; the first construction will not be used later, but we include it as it provides a larger class of families for which Theorem \ref{thm:main} is asymptotically sharp. It is also very interesting that a random construction gives almost the same (almost sharp) bound as a deterministic one! 

\label{constructions}
\begin{example}[Random construction]
\label{exam:rand}
    Fix $k,s \in \mathbb{N}$. For $n \geq C_0ks^3$ (for an appopriate absolute constant $C_0>0$), we take $\mathcal{F} \subset {[n] \choose k}$ to be the (random) family that is the disjoint union of:
    \begin{enumerate}
        \item the ``core part'', which consists of all sets $X \in {[n] \choose k}$ with $\abs{X \cap [s + 1]}\ge 2$; and 
        \item the ``random part'' where for each set $X \in {[n] \choose k}$ with $\abs{X \cap [s + 1]} = 1$, we include $X$ with probability $\lambda/(s+1)$ (independently, for each such $X$).
    \end{enumerate}
The size of core part is always 
    \[\sum_{i=2}^{s+1}\binom{s+1}{i} \binom{n - s- 1}{k-i},\] and the size of the random part follows a binomial distribution, whose expected value is \[\lambda \binom{n - s - 1}{k-1}.\] 
    Hence, $\mathbb{E}[|\mathcal{F}|]$ is equal to 
    \begin{equation}\label{eq:desired-size}\lambda \binom{n - s - 1}{k-1} + \sum_{i=2}^{s+1}\binom{s+1}{i} \binom{n - s- 1}{k-i},\end{equation}
    the expression that appears in (\ref{Fsize}). When the mean of a binomial distribution is an integer, the mean coincides with the median (an elementary proof of this can be found in \cite{median}), so with probability at least $1/2$, the family $\cat F$ will have size at least (\ref{eq:desired-size}). 

    Now we bound the maximum degree of the induced subgraph $K(n,k))[\mathcal{F}]$. First consider vertices in the random part. Let $X$ be a vertex in the random part. We may assume $X \cap [s + 1] = \{1\}$. In the Kneser graph, this vertex $X$ has $\binom{n-k-s}{k-1}$ neighbours $Y \in \cat D_i$ ($i = 2, 3, \ldots, s+1$) such that $\abs{Y \cap [s+1]} = 1$, so we expect $\frac{\lambda}{s+1} \binom{n-k-s}{k-1}$ such neighbours in $\cat F \cap \cat D_i$, so $\frac{s\lambda}{s+1} \binom{n-k-s}{k-1}$ such neighbours in the random part of $\cat F$. We also need to take into account the neighbours of $X$ in the core part,
    so the degree of $X$ in $\cat F$ has expectation 
    \begin{equation} \frac{s\lambda}{s+1}\binom{n - k - s}{k-1} + \sum_{i=2}^{s} \binom{s}{i}\binom{n-k-s}{k-i}. \label{ex:meandeg}\end{equation} 

    It turns out that any core vertex $X$, say with $\abs{X \cap [s + 1]} = x \ge 2$, will have smaller degree than (\ref{ex:meandeg}). Indeed, if $s=1$ then such a vertex $X$ has degree zero, whereas for $s>1$ its degree satisfies
    \begin{align*} d_X & \leq \sum_{j=1}^{s+1-x}\abs{\left\{Y \in {[n] \choose k}:\ Y \cap X = \varnothing,\ \abs{Y \cap [s+1]}=j\right\}}\\
    & = \sum_{j=1}^{s+1-x} {s+1-x \choose j} {n-k-s-1+x \choose k-j}.\end{align*}
    Provided $n \geq C_0ks^3$ for some appropriate absolute constant $C_0$, the ratio between successive terms in the sum on the right-hand side is at most $1/(100 s^2)$, and comparing the leading term with that of (\ref{ex:meandeg}) we see that
    \begin{align*}
    \frac{(s+1-x){n-k-s-1+x \choose k-1}}{(s\lambda/(s+1)){n-k-s \choose k-1}} & \leq
    \frac{s^2-1}{s^2} \frac{{n-k \choose k-1}}{{n-k-s \choose k-1}}\\
    & \leq \frac{s^2-1}{s^2} \left(1+\frac{k-1}{n-2k-s+2}\right)^s\\
    &\leq \frac{s^2-1}{s^2}\left(1+\frac{4k}{n}\right)^s\\
    & \leq \frac{s^2-1}{s^2}e^{4ks/n}\\
    & \leq \left(1-1/s^2\right)(1+8ks/n)\\
    & \leq \left(1-1/s^2\right)\left(1+1/\left(2s^2\right)\right)\\
    & \leq 1-1/\left(2s^2\right),
    \end{align*}
 so $d_X$ is smaller than 
    \[\frac{s\lambda}{s+1} \binom{n - k - s}{k-1}\]
    provided $n \geq C_0ks^3$.

    Finally we need to ensure that with probability greater than $1/2$, the maximum degree is not far from the expected degree (\ref{ex:meandeg}).

    If $\abs{X \cap [s + 1]} = 1$, then the number of neighbours of $X$ included in the random part of $\cat F$ is a binomially distributed variable $B \sim \operatorname{Bin}\left(s\binom{n-k-s}{k-1}, \frac{\lambda}{s+1}\right)$. By the Chernoff bound given in Theorem A.1.4 of \cite{chernoff}, 
    \begin{align*} \PP \left(B > s\binom{n-k-s}{k-1} \frac{\lambda}{s+1} (1+\delta) \right)& < \exp\left(-2 s\binom{n-k-s}{k-1} \left(\frac{\lambda}{s+1}\right)^2 \delta^2 \right) \\
    &
    \le \exp\left(-\frac{s}{2} \binom{n-k-s}{k-1} \delta^2 \right).  \end{align*}
    Taking the union bound over all such vertices $X$ (there are $(s+1) \binom{n-s-1}{k-1}$ of them), we have 

    \begin{align*}&\phantom{ \le 1} \PP\left( \exists X\in \binom{[n]}{k} \text{ (not necessarily in } \cat F) \right. \text{ such that } \abs{X \cap [s+1]} = 1 \\ 
    &\left.\hspace{2.5em}\text{ and } X \text{ has more than }  \frac{s \lambda}{s+1} \binom{n-k-s}{k-1} (1+\delta) \text{ neighbours in the random part of } \cat F\right) \\
    &\le (s+1)\binom{n-s-1}{k-1} \exp \left(-\frac{s}{2} \binom{n-k-s}{k-1} \delta^2\right).
    \end{align*}
 So if we pick $\delta = \left(s \binom{n-k-s}{k-1}\right)^{-1/2+\eps} $ for some constant $\eps > 0$, then for $n \geq C_0(k+s)$ for an appropriate absolute constant $C_0>0$, the probability of having some $X$ with such large degree in the random part of $\cat F$ is less than $1/2$. In other words, with probability more than $1/2$ the maximum degree of $\cat F$ is not far from the expected degree (\ref{ex:meandeg}) and satisfies 
    \begin{align*} \Delta(\cat F) &\le \frac{s\lambda}{s+1}\left(1+ \delta \right)\binom{n - k - s}{k-1}  + \sum_{i=2}^{s} \binom{s}{i}\binom{n-k-s}{k-i} \\
&\le \left(\frac{s \lambda}{s+1} + s \cdot \delta\right)\binom{n - k - 1}{k-1}. 
\end{align*}
(Note that for the last inequality, the identity
\begin{equation}\binom{n-k-1}{k-1} = \sum_{i=0}^{s-1}\binom{s-1}{i} \binom{n-k-s}{k-i-1}
\numberthis\label{eq:venn}\end{equation}
    is used.)
    
So with positive probability, $\abs{\cat F}$ is large enough but $\Delta (\cat F)$ is not too large, and we can remove any extra vertices from $\cat F$ to obtain a family of size parameter $\lambda$ and maximum degree at most \[\left(\frac{s \lambda}{s+1} + s \cdot \delta\right)\binom{n - k - 1}{k-1}.\]

\end{example}

\begin{example}[Explicit construction]
\label{explicitconstruction}
Let $(n,k,s) \in \mathbb{N}^3$ with $n \ge 12ks$. Construct $\mathcal{F} \subset {[n] \choose k}$ as follows. For each $X \in {[n] \choose k}$ with $\abs{X \cap [s + 1]} = 1$, include $X$ in $\cat F$ if and only if 
$X \subseteq [t]$, where $t\le n$ is the least integer that satisfies 
\[\binom{t - s -1}{k - 1} \ge \frac{\lambda}{s+1}  \binom{n - s -1}{k - 1}. \]
We also include in $\mathcal{F}$ all those sets $X \in {[n] \choose k}$ with $\abs{X \cap [s + 1]} \ge 2$ (the ``core part''). One can easily check that $\cat F$ has size-parameter at least $\lambda$.

Note that $\lambda/(s+1) \ge 1/2$, so 
\begin{equation} t-s-1 \ge \frac{1}{2}  (n-s-1)> \frac{1}{3} n.
\label{bislarge}
\end{equation}
Define $p:=k/n$, and note that $p \leq 1/(12s)$. By choice of $t$, we have
\begin{align*}\binom{t - s - 1}{k - 1} \ge \frac{\lambda}{s+1}  \binom{n - s -1}{k - 1} &> \binom{t - s - 2}{k - 1} \\
&=\binom{t - s - 1}{k - 1} \cdot \frac{t-s-k}{t- s - 1} \\
&=\binom{t - s - 1}{k-1} \cdot \left(1 - \frac{k-1}{t-s-1}\right) \\
&> \binom{t - s - 1}{k-1} (1-3p) & \text{ by (\ref{bislarge}).}
\end{align*}
Fix $X\in \cat F$ with $X\cap [s+1] = \{1\}$. The number of sets $Y\in \cat F$ such that $Y\cap [s+1] = \{2\}$ and $Y \cap X = \emptyset$ is equal to
\[ \binom{t - s - k }{k - 1}. \]
We have
\begin{align*} \left.\binom{t - s - k }{k - 1} \middle/ \binom{n - s - k}{k-1}\right. & <
\left.\binom{t - s - 1 }{k - 1} \middle/ \binom{n - s - 1}{k-1}\right.\\
& < \frac{\lambda}{s+1}\cdot \frac{1}{1-3p}\\
&\le (1+4p) \frac{\lambda}{s+1},\end{align*}
using the fact that $p \le 1/12$. Therefore, the degree $d_X$ of $X$ (also taking into account its neighbours in the core part) satisfies
\begin{align*} 
d_X 
&= s \binom{t - s - k }{k - 1} + \sum_{i=2}^{s} \binom{s}{i}\binom{n-k-s}{k-i}\\
&< \frac{s\lambda}{s+1} (1+4p) \binom{n - k - s}{k-1}  + \sum_{i=2}^{s} \binom{s}{i}\binom{n-k-s}{k-i} \\
&\le \left(\frac{s\lambda}{s+1} + 4sp\right) \binom{n-k-1}{k-1}. \numberthis \label{upperbound}
\end{align*}
(Note that the last inequality uses \cref{eq:venn} again.)
It is clear that $d_X$ is the maximum degree of $\cat F$ as any core vertex $Z\in \cat F$ with $\abs{Z \cap [s+1]}\ge 2$ has smaller degree, similarly to in the previous example. We can remove elements of $\cat F$ if there are too many (clearly this does not increase $\Delta(\cat F)$), so as to obtain exactly the right size-parameter.
\end{example}

This construction is motivated by the fact that picking sets in each star according to the colex order makes sets in different stars intersect more often and is therefore helpful for reducing the maximum degree of the induced subgraph. We conjecture that to minimize the maximum degree among vertex sets of the same size, we can take families similar to those in \cref{explicitconstruction}. To state our conjecture, we need one more piece of notation.

\begin{notation}
\label{techset}
    For a family $\cat F\subseteq \binom{[n]}{k}$ and $I \subseteq J \subseteq [n]$, we write
    \[\cat F^I_J = \{F \setminus J: F\in \cat F, F \cap J = I\} \subseteq \binom{[n]\setminus J}{k - \abs{I}}.\] 
    Its elements correspond to elements of $\cat F$ that intersect with $J$ in the prescribed way.
\end{notation}

\begin{conjecture}[Sparse minimizer]
\label{conj:sparse}
Suppose $n$ is sufficiently large compared to $k$ and $s$, and $s \le \lambda \le s+1$. then there exists a minimizer $\cat F \subseteq \binom{[n]}{k}$ with size parameter $\lambda$ (whose maximum degree $\Delta(\cat F)$ is minimum among all such families of the same size) with the following properties:
\begin{enumerate}[label=(\roman*)]
    \item \label{contained in stars}$\cat F \subseteq \cat D_1 \cup \cat D_2 \cup \cdots \cup \cat D_{s+1}$;
    \item \label{contains intersections}$\cat D_i \cap \cat D_j \subseteq \cat F$ for all $i, j \in [s+1]$; and
    \item \label{almost equal}$\abs{\cat F_{[s+1]}^{\{i\}} \mathbin{\triangle}\cat F_{[s+1]}^{\{j\}}} \le 1 $ for all $i, j \in [s+1]$.
\end{enumerate}
\end{conjecture}

\cref{equalitycase} shows that \cref{conj:sparse} is true when $\lambda = s$ is an integer, because in this case, any minimizer must be a union of $s$ stars. 

If we assume that $\cat F$ satisfies \ref{contained in stars} and \ref{contains intersections}, then the problem reduces to choosing the $\cat F_{[s+1]}^{\{i\}}$'s optimally, which are $(s+1)$ subfamilies of $\binom{[n] \setminus [s+1]}{k-1}$ with sizes summing to $\lambda \binom{n - s - 1}{k-1}$. If we further assume \ref{almost equal}, then every $\cat F_{[s+1]}^{\{i\}}$ is almost equal to the same subfamily  $\cat F' \subseteq \binom{[n] \setminus [s+1]}{k-1}$ with density about $\lambda/(s+1)$ and the problem reduces to minimizing the maximum degree of $\cat F'$, so we have come back to the original problem but in the dense case (since $\lambda/(s+1)\ge 1/2$), as noted on p. 659 of \cite{edgenum}. The way we chose $\cat F'$ in \cref{explicitconstruction} suggests the following conjecture.

\begin{conjecture}[Dense minimizer]
\label{conj:dense}
Suppose $\binom{n}{k} / 2 \le m \le \binom{n}{k}$ (so the size parameter $\lambda$ can be as large as $\Theta(n)$ for fixed $k$), and let $t$ be the smallest natural number such that $\binom{t}{k}\ge m$, then there is some
 \[ \cat F \subseteq \binom{[t]}{k}\] of size $\abs{\cat F} = m$ such that $\Delta(\cat F)$ is minimum over all subfamilies of $\binom{[n]}{k}$ of size $m$.
\end{conjecture}

When $m$ happens to be of the form $\binom{t}{k}$, the conjectured optimal family is $\cat F = \binom{[t]}{k}$, which is an initial segment of colex. In this case the conjecture  would follow from the corresponding statement that initial segments of colex minimizes the number of edges (such as Corollary 2.4 in \cite{edgenum} for some range of parameters), because minimizing the number of edges in addition to being regular is enough to guarantee minimizing the maximum degree.

However, when $m$ is not of the special form, the initial segment of colex is not necessarily optimal, just as the initial segment of lex is not optimal for $m = \binom{n-1}{k-1} + 1$.

\begin{example}
    Let $m = \binom{n-1}{k} + 1$. The initial segment of colex is 
    \[\cat F_1 = \binom{[n-1]}{k} \cup \{\{1, 2, \ldots, k-1, n\}\},\]
    so the vertex $\{1, 2, \ldots, k-1, n\}$ has degree $\binom{n-k}{k} = \Delta(K(n, k))$ in $K(n,k)[\mathcal{F}_1]$, i.e., the maximum possible value! We can easily improve on this: if we pick each vertex of $K(n, k)$ independently with probability $\left.m\middle/\binom{n}{k} \right. \approx 1-k/n $, then similarly to in \cref{exam:rand}, we have positive probability of getting a family $\cat F_2$ satisfying both $|\mathcal{F}_2| \geq |\mathcal{F}_1|$ and
    $$\Delta(\cat F_2) = \left(1 - k/n + O\left(\binom{n-k}{k}^{-1/2 + \eps} \right)\right) \binom{n-k}{k} < \Delta(\mathcal{F}_1).$$
\end{example}

This example also shows that for the maximum degree problem there is no counterpart of Lemma 2.3 in \cite{edgenum} which says if $\cat F$ minimizes the edge number among families of equal size, then the same holds for $\binom{[n]}{k} \setminus \cat F$. Indeed the complement of $\cat F_1$ is a subset of the star $\cat D_n$, so it has maximum degree 0 and therefore (trivially) minimizes the maximum degree over all sets of the same size, whereas, by the above argument, $\mathcal{F}_1$ does not minimize the maximum degree over all sets of the same size. 

\section{Background on singular values and the Expander Mixing Lemma}

Our proof of Theorem \ref{thm:main} relies on the Expander Mixing Lemma, which allows us to estimate the number of edges between two large, disjoint subsets of the vertex-set of a regular graph whose nontrivial eigenvalues are small in absolute value.

In this section, we outline some standard results we need from the spectral theory of biregular bipartite graphs, and more specifically from the spectral theory of (bipartite) Kneser graphs. 

\subsection{Singular values of biregular bipartite graphs}
\begin{definition}[Singular values of a linear map / bipartite graph]
Let $V_1$ and $V_2$ be finite-dimensional inner product spaces, and let $A:V_2 \to V_1$ be a linear map (also viewed as a matrix). Then a \emph{singular value decomposition} of $A$ is an expression of $A$ as a linear combination of rank-one linear maps
\[A = \sum_{i = 1}^{r} \sigma_i \vc{u}_i \vc{v}_i^T,\] where $\{\vc u_i\}_{i = 1}^{\dim V_1}$ is an orthonormal basis of $V_1$, $\{\vc v_i\}_{i = 1}^{\dim V_2}$ is an orthonormal basis of $V_2$,  $r = \min \{\dim V_1, \dim V_2\}$, and $\sigma_1 \ge \sigma_2 \ge \cdots \ge \sigma_r \geq 0$ are non-negative real numbers. The $\sigma_i$'s are called the \emph{singular values} of $A$ and are independent of the choice of orthonormal bases; moreover, the transpose $A^T$ (or dual linear map) has the same singular values as $A$.

Let $G$ be a bipartite graph with bipartition $U \sqcup V$. The \emph{singular values} of $G$ are the singular values of its bipartite adjacency matrix $A$, viewed as an linear operator $A: \RR^V \to \RR^U$ defined by $(A(\vc x))_u = \sum_{v \sim u}  x_v \quad \forall u \in U$. (Note that here, the spaces $\RR^U$ and $\RR^V$ are equipped with the standard inner product induced by the counting measure, making the bases $\{ \delta_u: u \in U\}$ and $\{ \delta_v: v \in V\}$ orthonormal. For example, the all-ones vector $\mathbbm 1_U = \sum_{u \in U} \delta_u$ has norm $\sqrt{\abs {U}}.$)
\end{definition}

The following is well-known; we provide a proof for completeness.

\begin{lemma}[Highest singular value of a biregular bipartite graph]
\label{lem:bireg}
Let $G$ be a bipartite graph with at least one edge and with bipartition $U \sqcup V$, such that every vertex in $U$ has degree $d_U$ and every vertex in $V$ has degree $d_V$. Let $A:\RR^V \to \RR^U$ denote the bipartite adjacency matrix of $G$. Then $A$ has largest singular value $\sigma_1 = \sqrt{d_V d_U}$ with $\mathbf{v}_1 = 1_V/\sqrt{|V|}$ and $\mathbf{u}_1 = 1_U/\sqrt{|U|}$ being corresponding singular vectors.
\end{lemma} 

\begin{proof}
    The singular values of $A$ are precisely the (non-negative) square roots of eigenvalues of $A^TA$ (which is symmetric and positive semidefinite). Observe that $(A^TA)_{v_1 v_2}$ counts the number of length-2 paths from $v_1$ to $v_2$ in $G$, so in the matrix $A^T A$, all $\abs{V}$ column sums and all $\abs{V}$ row sums are equal to $d_V d_U$, so $A^T A$ is $d_V d_U$ times a doubly stochastic matrix, and so its largest eigenvalue is $d_V d_U$. Hence $\sigma_1 = \sqrt{d_V d_U}$. It is easy to check that $A\left(1_V/\sqrt{|V|}\right) = \sqrt{|U|} 1_U$, verifying the last claim.

\end{proof}

\begin{notation}\label{innerprod}
For vectors in $\vc x, \vc y \in \RR^U$ (or, interchangeably, functions $f, g: U \to \RR$), we write their {\em inner product} as
\[\inner{\vc x, \vc y} := \sum_{u \in U}  x_u y_u \quad \left(\text{or } \inner{f, g} = \sum_{u \in U} f(u) g(u)\right),\]
and the induced Euclidean norm is defined by $\|\vc x\|_2^2 := \langle \vc x, \vc x\rangle$, $\|f\|_2^2 = \langle f,f\rangle$. Note that this normalisation differs from that in some texts, but it will be convenient for us, as we will work with $U$'s of different sizes.

\noindent For functions $f:U \to \mathbb{R}$, we also write 
\[\EE(f) = \frac{1}{\abs{U}}\sum_{u \in U} f(u),\]
so for example $\EE(f^2) = \norm{f}_2^2/\abs{U}.$
\end{notation}

\begin{theorem}[Expander Mixing Lemma, biregular case; see e.g.\cite{bipartexpand} Lemma 8, or \cite{bipartexpandhist} Theorem 3.1.1.]
Let $G$ be a bipartite graph with at least one edge and with bipartition $U \sqcup V$, such that every vertex in $U$ has degree $d_U$ and every vertex in $V$ has degree $d_V$. Let $A:\RR[V] \to \RR[U]$ denote the adjacency matrix of $G$. Let $\sigma_1$ be its largest singular value and $\sigma_2$ its second-largest (counting with multiplicity). Then for any subsets $X \subseteq U$ and $Y \subseteq V$,
with densities $\alpha = \abs{X} / \abs{U}$ and $\beta = \abs{Y} / \abs{V}$, we have 
\[\abs{\frac{e(X, Y)}{e(U, V)} - \alpha \beta} \le \frac{\sigma_2}{\sigma_1} \sqrt{\alpha (1-\alpha) \beta (1-\beta)} \le  \frac{\sigma_2}{\sigma_1} \sqrt{\alpha\beta}.\]

\end{theorem}
\begin{proof}
Let $\mathbbm 1_X:U \to \RR$ and $\mathbbm 1_Y:V \to \RR$ be the indicator functions of $X$ and $Y$ respectively; then $e(X, Y) = \inner{\mathbbm 1_X, A \mathbbm 1_Y}$. Let \[A = \sum_{i = 1}^{r} \sigma_i \vc{u}_i \vc{v}_i^T\] be a singular value decomposition where
$\{\vc u_i\}_{i = 1}^{\abs{U}}$ is an orthonormal basis for $\RR^ U$,
$\{\vc v_i\}_{i = 1}^{\abs{V}}$ an orthonormal basis for $\RR^V$,
$r = \min (\abs{V}, \abs{U})$, 
and singular values $\sigma_1 \ge \sigma_2 \ge \cdots \ge \sigma_r \ge 0$,
and $\vc u_1 = \mathbbm 1_U / \sqrt{\abs{U}}$, $\vc v_1 = \mathbbm 1_V / \sqrt{\abs{V}}$  are singular vectors corresponding to $\sigma_1$. In these bases,  $\mathbbm 1_X$ and $\mathbbm 1_Y$ decompose as:
\[\mathbbm 1_X = \alpha \mathbbm 1_U + x_2 \vc u_2 + \cdots + x_{\abs{U}} \vc u_{\abs{U}},\]
\[\mathbbm 1_Y = \beta \mathbbm 1_V + y_2 \vc v_2 + \cdots + y_{\abs{V}} \vc v_{\abs{V}},\]
where $x_i = \inner{\mathbbm 1_X, \vc u_i}$, and $y_i = \inner{\mathbbm 1_Y, \vc v_i}$. Now,
\[A \mathbbm 1_Y = \beta d_U \mathbbm 1_U + \sigma_2 y_2 \vc u_2 + \cdots + \sigma_r y_r \vc u_r,\]
so 
\[\inner{\mathbbm 1_X, A \mathbbm 1_Y} = \alpha \beta d_U \abs{U} + \sum_{i = 2}^r \sigma_i x_i y_i, \]
so 
\begin{align*}
    \abs{e(X, Y) - \alpha \beta e(U, V)} 
    &= \abs{\sum_{i = 2}^r \sigma_i x_i y_i} \\
    &\le \sum_{i=2}^r \sigma_i \abs{x_i}\abs{y_i} \\
    &\le \sigma_2 \left(\sum_{i = 2}^r x_i^2\right)^{1/2} \left(\sum_{i = 2}^r y_i^2\right)^{1/2} \\
    &\le \sigma_2 \left(\norm{\mathbbm 1_X}_2^2 - \norm{\alpha \mathbbm 1_U}_2^2 \right)^{1/2} \left(\norm{\mathbbm 1_Y}_2^2 - \norm{\beta \mathbbm 1_V}_2^2 \right)^{1/2} \\
    &= \sigma_2  \left(\alpha \abs{U} - \alpha^2 \abs{U} \right)^{1/2} \left(\beta \abs{V} - \beta^2 \abs{V} \right)^{1/2} \\
    &= \sigma_2 \sqrt{\alpha (1-\alpha) \beta (1-\beta) \abs{U} \abs{V}}.\numberthis \label{ineq:expander}
\end{align*}
Recall that $e(U, V) = \abs{U} d_U = \abs{V} d_V$, so $e(U, V) = \sqrt{\abs{U} \abs{V} d_U d_V} =\sigma_1 \sqrt{\abs{U}\abs{V}}$. So dividing (\ref{ineq:expander}) by $e(U, V)$, we obtain  

\[\abs{\frac{e(X, Y)}{e(U, V)} - \alpha \beta} \le \frac{\sigma_2}{\sigma_1} \sqrt{\alpha (1-\alpha) \beta (1-\beta)},\]
as required.
\end{proof}

Multiplying both sides of the inequality in the above theorem by $e(U, V)/\abs{X}$, which is equal to $d_U \abs{U} /\abs{X} = d_U / \alpha$, we obtain the following.

\begin{corollary}
\label{sqrtlb}
Let $G$ be a non-empty bipartite graph with vertex partition $U \sqcup V$, biregular of degrees $d_U$ and $d_V$, and let $X \subseteq U$ and $Y \subseteq V$ be subsets of densities $\alpha = \abs{X} / \abs{U}$ and $\beta = \abs{Y} / \abs{V}$, and $\alpha >0$, then
\[\abs{\frac{e(X, Y)}{\abs{X}} - d_U \beta} \le \frac{\sigma_2}{\sigma_1} d_U \sqrt{\frac{\beta (1-\beta)}{\alpha}}\le \frac{\sigma_2}{\sigma_1} d_U \sqrt{\frac{\beta}{\alpha}},\]
where $\sigma_1 \ge \sigma_2$ are the two largest singular values of $G$. In particular, the maximum degree $\Delta (G)$ is at least the average degree in $X$, so 
    \[\Delta (G) \ge \frac{e(X, Y)}{\abs{X}} \ge d_U \left(\beta - \frac{\sigma_2}{\sigma_1}\sqrt\frac{\beta (1-\beta)}{\alpha}\right) \ge d_U \left(\beta - \frac{\sigma_2}{\sigma_1}\sqrt\frac{\beta}{\alpha}\right). \]
\end{corollary}

\subsection{Spectra of Kneser graphs and bipartite Kneser graphs}

In \cite{knesereigen} Theorem 13, Lov\'{a}sz gave the eigenspace decomposition of $K(n, k)$, which also allows us to prove the bound on the singular values for the bipartite Kneser graph between $\binom{[n]}{k}$ and $\binom{[n]}{l}$ for $k \neq l$.  We describe the decomposition (and prove some simple properties thereof), in the next lemma.

\begin{lemma}[Decomposition implicit in \cite{knesereigen}, Theorem 13]
\label{lem:decomp}
    For each $0 \leq i \le j \leq n/2$, define $\iota_{ij}:\RR^{\binom{[n]}{i}} \to \RR^{\binom{[n]}{j}} $ by
    \[\left(\iota_{ij}(\vc x)\right)_J = \EE( x_I | I \subseteq J),\]
     i.e.\ the average of the $ x_I$'s as $I$ ranges over the $\binom{j}{i}$ size-$i$ subsets of $J$. For each $i \geq 1$, define the subspace $P_i \le \RR^{\binom{[n]}{i}}$ to be the image of the map $\iota_{(i-1)i}$, i.e.,
    \[P_i := \iota_{(i-1)i}\left(\RR^{\binom{[n]}{i-1}}\right),\]
    and define $P_0 := \{0\}$. Let $Q_i := P_i^\perp$
    be the orthogonal complement of $P_i$ in $\RR^{\binom{[n]}{i}}$, and for each $i \leq k \leq n/2$ write
    \[Q_i^{(k)} := \iota_{ik} (Q_i).\]
Then for each $i \leq k \leq n/2$, we have
\begin{equation}\label{eq:dim-equality}\dim Q_i^{(k)} = \dim Q_i = \binom{n}{i} - \binom{n}{i-1},\end{equation}
and the vector space $\RR^{\binom{[n]}{k}}$ decomposes into a direct sum: 
    \begin{equation}\label{eq:orthog-decomp} \RR^{\binom{[n]}{k}} = \bigoplus_{i = 0}^k Q_i^{(k)}.\end{equation}
    Furthermore, $Q_0^{(k)}$ consists of the constant vectors in $\mathbb{R}^{{[n] \choose k}}$, and 
    \begin{equation}\label{eq:useful}Q_1^{(k)} = \operatorname{Span}\{\mathbbm 1_{j \in I}-k/n:\ j \in [n]\}.\end{equation}
\end{lemma}
\begin{proof}
    Firstly, as noted by Lov\'asz \cite{knesereigen} and first proven by Gottlieb \cite{gottlieb}, for each $0 \leq i \le j \leq n/2$ the map $\iota_{ij}:\RR^{\binom{[n]}{i}} \mono \RR^{\binom{[n]}{j}}$ is an injection, and moreover these injections form a directed system, meaning that $\iota_{ik} = \iota_{jk} \circ \iota_{ij}$ for all $0 \leq i \le j \le k \leq n/2$) (this follows from the tower rule for conditional expectation).

Since $\iota_{(i-1)i}$ is an injection, we have $\dim(P_i) = {n \choose i-1}$ for all $i$, using the convention that ${n \choose -1}=0$. Therefore, we have $\dim(Q_i) = {n \choose i} - {n \choose i-1}$ for all $i$. Since $\iota_{ik}$ is an injection, we have $\dim(Q_i^{(k)})=\dim(Q_i) = {n \choose i}-{n \choose i-1}$ for all $i$, proving (\ref{eq:dim-equality}).

To prove (\ref{eq:orthog-decomp}), we use induction on $k$. The $k=0$ case is obvious. Let $k \geq 1$, and assume that
\[\RR^{\binom{[n]}{k-1}} = \bigoplus_{i = 0}^{k-1} Q_i^{(k-1)}.\]
Then
\begin{align*}
    \RR^{\binom{[n]}{k}} &= P_k \oplus Q_k & \text{since }Q_k \text{ is the orthogonal complement of }P_k \\
    &= \iota_{(k-1)k} \left(\RR^{\binom{[n]}{k-1}}\right) \oplus Q_k \\
    &= \iota_{(k-1)k} \left(\bigoplus_{i = 0}^{k-1} Q_i^{(k-1)} \right) \oplus Q_k &\text{by the induction hypothesis} \\
    &= \bigoplus_{i = 0}^{k-1} \iota_{(k-1)k} \left( Q_i^{(k-1)} \right) \oplus Q_k &\text{since }\iota_{(k-1)k} \text{ is injective} \\
    &= \bigoplus_{i = 0}^{k} Q_i^{(k)}. 
\end{align*}
This proves (\ref{eq:orthog-decomp}).

Since $P_0 = \{0\}$, we have $Q_0 = \mathbb{R}^{{[n] \choose 0}}$, and therefore $Q_0^{(k)}$ consists of the constant vectors in $\mathbb{R}^{[n] \choose k}$, as claimed. Moreover, $P_1$ consists of the constant vectors in $\mathbb{R}^{{[n] \choose 1}}$, so $Q_1 = \{\mathbf{x} \in \mathbb{R}^{{[n] \choose 1}}:\ \sum_{i=1}^{n}x_{\{i\}}=0\}$, so the vectors $\{\mathbf{x}(j):\ j \in [n]\}$ defined by $(\mathbf{x}(j))_{\{i\}} = k(\mathbbm 1_{\{i=j\}}-1/n)$ span $Q_1$, and $\iota_{1k}(\mathbf{x}(j)) = \mathbf{y}(j)$, where $(\mathbf{y}(j))_I = \mathbbm 1_{j \in I}-k/n$ for all $I \in {[n] \choose k}$,
    so (\ref{eq:useful}) holds.
    
\end{proof}

We will see shortly that (\ref{eq:orthog-decomp}) is an orthogonal decomposition. But next, we observe that the bipartite Kneser adjacency matrix acts nicely on the decomposition (\ref{eq:orthog-decomp}).

\begin{lemma}[Bipartite Kneser adjacency matrix under the decomposition (\ref{eq:orthog-decomp}), modification of \cite{knesereigen} proof of Theorem 13]
\label{lem:nice}
    Let $k \le l \le n/2$, and $U = \binom{[n]}{k}$ and $ V = \binom{[n]}{l} $. Let $A: \RR^{\binom{[n]}{k}} \to \RR^{\binom{[n]}{l}}$ be the bipartite adjacency matrix for the bipartite Kneser graph between $U$ and $V$. Then $A(Q_i^{(k)}) \subset Q_i^{(l)}$ for all $i \leq k$, and in fact for any $\vc y \in Q_i^{(k)}$, we have
    \begin{equation}A \vc y =  (-1)^i \binom{n-l-i}{k-i} \iota_{kl} (\vc y) \in Q_i^{(l)}.
    \end{equation}
\end{lemma}
\begin{proof}
    
    Let $\vc x $ be an arbitrary element of $Q_i$ and consider the action of the bipartite Kneser adjacency matrix $A$ sending $\vc y: = \iota_{ik}(\vc x)\in Q_i^{(k)} \subseteq \RR^{\binom{[n]}{k}}$ to $\RR^{\binom{[n]}{l}}$. Fix $L \in \binom{[n]}{l}$. Then 
    \begin{align*}
        (A \iota_{ik}(\vc x) )_L
        &= \sum_{\substack{K:\ \abs{K} = k, \\ K \cap L = \varnothing}} (\iota_{ik}(\vc x))_K \\
        &= \sum_{\substack{K:\ \abs{K} = k, \\ K \cap L = \varnothing}} \EE_{\abs{I} = i} (x_I |I \subseteq K)\\
        &= \sum_{\substack{I:\ \abs{I} = i, \\ I \cap L = \varnothing}} \sum_{\substack{K:\ \abs{K} = k, \\ I \subseteq K \subseteq [n] \setminus L}} \frac{1}{\binom{k}{i}}  x_I \\
        &= \frac{\binom{n-l-i}{k-i}}{\binom{k}{i}} \sum_{\substack{I:\ \abs{I} = i, \\ \abs{I \cap L} = 0}}  x_I. \numberthis \label{eq:S0}
    \end{align*}
For each $j \leq i$, let \[S_j: = \sum_{\substack{I:\ \abs{I} = i, \\ \abs{I \cap L} = j}}  x_I;\]
we need to find $S_0$. For $i \ge 1$,  the subspace $P_i \le \RR^{\binom{[n]}{i}}$ is spanned by the indicators for the event $I \supset T$ for each $T$ of size $\abs{T} = i-1$. Denote these indicator vectors by $\vc v_T$ (so $(\vc v_T)_I$ is 1 if $T \subset I$, and is 0 otherwise). Since $\vc x \in Q_i$ is orthogonal to $ P_i = \operatorname{Span}\left\{\vc v_T: T \in \binom{[n]}{i-1}\right\}$, we know that for every $T \in \binom{[n]}{i-1}$, we have
    \[\sum_{\substack{I:\ \abs{I} = i,\\ I \supseteq T}} x_I = \inner{\vc x, \vc v_T} = 0.\]
If we sum this equality over all $T$'s with $\abs{T} = i-1$ and $\abs{T \cap L} = j$, then only those $I$'s with $\abs{I \cap L} = j$ or $j+1$ can appear, and moreover each of those $I$'s with $\abs{I \cap L} = j$ appears $i-j$ times (when $T$ equals $I$ minus one element in $I \setminus L$), and each of those $I$'s with $\abs{I \cap L} = j+1$ appears $j+1$ times (when $T$ equals $I$ minus one element in $I \cap L$),
    so
    \[(i-j) S_j + (j+1) S_{j+1} = 0.\]
Now we can solve iteratively, obtaining
    \[S_0 = \left(-\frac{1}{i}\right)S_1 = \left(-\frac{1}{i}\right)\left(-\frac{2}{i-1}\right) S_2 = \cdots = \left(-\frac{1}{i}\right)\left(-\frac{2}{i-1}\right) \cdots \left(-\frac{i}{1}\right) S_i = (-1)^i \sum_{\substack{I:\ \abs{I} = i,\\ I \subseteq L}}  x_I.\]
    Substituting this back into (\ref{eq:S0}), we obtain 
    \[(A \iota_{ik}(\vc x) )_L = (-1)^i \binom{n-l-i}{k-i} \EE_{\abs{I} = i} (x_I | I \subseteq L) = (-1)^i \binom{n-l-i}{k-i} (\iota_{il}(\vc x))_L.\]
This holds for all $L \in \binom{[n]}{l}$, so we may conclude that
    \begin{align}A\vc y &= A \iota_{ik}(\vc x)\nonumber \\
    &= (-1)^i \binom{n-l-i}{k-i} \iota_{il}(\vc x) \nonumber\\
    &=   (-1)^i \binom{n-l-i}{k-i} \iota_{kl}(\iota_{ik}  (\vc x))\nonumber\\
    &=  (-1)^i \binom{n-l-i}{k-i} \iota_{kl}(\vc y) \in \iota_{kl} \left(Q_i^{(k)}\right) = Q_i^{(l)}, \label{eq:eigenval}
    \end{align}
    as required.
\end{proof}

\begin{lemma}[Eigenvalues of the Kneser graph, \cite{knesereigen} Theorem 13]
\label{eigenvalkneser}
Let $k \le n/2$. Let $A$ be the Kneser graph adjacency matrix on $\RR^{\binom{[n]}{k}}$. Then (\ref{eq:orthog-decomp}) is the eigenspace decomposition for $A$, with $Q_i^{(k)}$ corresponding to the eigenvalue 
\[(-1)^i \binom{n-k-i}{k-i}\]
In particular, (\ref{eq:orthog-decomp}) is an orthogonal decomposition.
\end{lemma}
\begin{proof}
 In the special case $k=l$, the map $\iota_{kl}$ is the identity, and the previous lemma shows that the $Q_i^{(k)}$'s are indeed eigenspaces of $A: \RR^{\binom{[n]}{k}} \to \RR^{\binom{[n]}{k}}$ corresponding to distinct eigenvalues, with $Q_i^{(k)}$ corresponding to the eigenvalue $(-1)^i \binom{n-k-i}{k-i}$. Therefore, the $Q_i^{(k)}$'s are pairwise orthogonal, being eigenspaces of a self-adjoint linear operator corresponding to distinct eigenvalues.
\end{proof}

We can now obtain our required bound on the singular values of the bipartite Kneser graph. 

\begin{lemma}[Singular values of the bipartite Kneser graph]
\label{singvalkneser}
Let $k < l \le n/2$, and $U = \binom{[n]}{k}$ and $ V = \binom{[n]}{l} $. Let $A: \RR^{\binom{[n]}{k}} \to \RR^{\binom{[n]}{l}}$ be the bipartite adjacency matrix for the bipartite Kneser graph between $U$ and $V$. Then (\ref{eq:orthog-decomp}) yields a singular value decomposition for $A : \RR^{\binom{[n]}{k}} \to \RR^{\binom{[n]}{l}}$, and if we let $\sigma_1\ge \sigma_2$ be the two largest singular values, then 
\[\frac{\sigma_2}{\sigma_1} \le \frac{k}{n-l} \le \frac{l}{n-k}.\]
\end{lemma}
\begin{proof}

    By Lemma \ref{lem:nice}, we may factorize $A: \RR^{\binom{[n]}{k}} \to \RR^{\binom{[n]}{l}}$ as $A = \iota_{kl} \circ D$, where the `scaling part' $D:\mathbb{R}^{{[n] \choose k}}\to \mathbb{R}^{{n \choose k}}$ acts on $Q_i^{(k)}$ as multiplication by $(-1)^i \binom{n-l-i}{k-i}$ (for each $i$), and $\iota_{kl}$ is $1/\binom{l}{k}$ times the bipartite adjacency matrix of another biregular graph between $\binom{[n]}{k}$ and $\binom{[n]}{l}$. Among the eigenvalues of $D$, $\binom{n-l}{k}$ has largest absolute value, and its eigenspace is 
    \[Q_0^{(k)} = \iota_{0k}\left(\RR^{\binom{[n]}{0}}\right) = \operatorname{Span}\{\mathbbm 1\},\]
    which is one-dimensional.
    By \cref{lem:bireg}, the all-ones vector $\mathbbm 1$ also attains the largest singular value of $\iota_{kl}$, i.e.,
    $$\frac{\norm{\iota_{kl}(\vc w)}_2}{\norm{\vc w}_2} \leq \frac{\norm{\iota_{kl}(\mathbbm 1)}_2}{\norm{\mathbbm 1}_2}\quad \forall \vc w \in \mathbb{R}^{{[n] \choose k}} \setminus \{0\},$$
    so $\mathbbm 1$ is also the highest singular vector for $A = \iota_{kl} \circ D$. All vectors $\vc v \in \mathbb{R}^{[n] \choose k} \setminus \{0\}$ that are orthogonal to $\mathbbm 1$ satisfy \[\norm{D \vc v}_2 \le \binom{n-l-1}{k-1} \norm{\vc v}_2\le \frac{k}{n-l}\frac{\norm{D \mathbbm 1}_2}{\norm{\mathbbm 1}_2} \norm{\vc v}_2,\]
    and therefore satisfy
    \begin{align*}
\frac{\norm{A \vc v}_2}{\norm{\vc v}_2} & = \frac{\norm{\iota_{kl}(Dv)}_2}{\norm{\vc v}_2}\\
& = \frac{\norm{\iota_{kl}(Dv)}_2}{\norm{D\vc v}_2}\cdot \frac{\norm{D\vc v}_2}{\norm{\vc v}_2}\\
& \leq \frac{\norm{\iota_{kl}(\mathbbm 1)}_2}{\norm{\mathbbm 1}_2} \cdot {n-l-1 \choose k-1},
\end{align*}
    and comparing this to
    \begin{align*}\frac{\norm{A \mathbbm 1}_2}{\norm{\mathbbm 1}_2} & = \frac{\norm{\iota_{kl}(D\mathbbm 1)}_2}{\norm{\mathbbm 1}_2}\\
& = \frac{\norm{\iota_{kl}(\mathbbm 1)}_2}{\norm{\mathbbm 1}_2} \cdot {n-l \choose k}
    \end{align*}
    yields
$$\frac{\sigma_2}{\sigma_1} \leq \frac{{n-l-1 \choose k-1}}{{n-l \choose k}} = \frac{k}{n-l} \leq \frac{l}{n-k},$$
as required.
\end{proof}

Recall the notation $\cat F_J^I$ (\cref{techset}), which will now be very useful. In the Kneser graph, we will need a lower bound on the number of edges between $\cat C = \cat F \cap \cat D_n$ and $\cat B = \cat F \setminus \cat D_n$. For any $c \in \cat C$ and $b \in \cat B$, we have $n \in c$ but $ n \notin b$, so $c \cap b = \varnothing \iff (c \setminus \{n\}) \cap b = \varnothing$, so we can remove $n$ from each $c \in \cat C$ to obtain $\cat C' = \cat F_{\{n\}}^{\{n\}}\subseteq \binom{[n-1]}{k-1}$, and count the edges between $\cat C'\subseteq \binom{[n-1]}{k-1}$ and $\cat B = \cat F_{\{n\}}^\varnothing\subseteq \binom{[n-1]}{k}$ in the bipartite Kneser graph. In this setting, we can combine \cref{sqrtlb} with \cref{singvalkneser} to obtain the following.
\begin{corollary}[Expander Mixing Lemma applied to a star and its complement]
    \label{kneserexpander}
    Let $n \ge 2k+1$ and $\cat F \subseteq \binom{[n]}{k}$. Suppose $\cat F_{\{n\}}^{\{n\}} \subseteq \binom{[n-1]}{k-1}$ has size $\gamma\binom{n-1}{k-1} $, and $\cat F_{\{n\}}^\varnothing\subseteq \binom{[n-1]}{k} $  has size $\beta\binom{n-1}{k}$. Then 

    \[\frac{e(\cat F_{\{n\}}^{\{n\}}, \cat F_{\{n\}}^\varnothing)}{\abs{\cat F_{\{n\}}^{\{n\}}}}
    \ge \binom{n-k}{k} \left(\beta - \frac{k}{n-k} \sqrt{\frac{\beta}{\gamma}}\right) \quad\text{ if } \gamma > 0,
    \]
    and
    \[\frac{e(\cat F_{\{n\}}^{\{n\}}, \cat F_{\{n\}}^\varnothing)}{\abs{\cat F_{\{n\}}^\varnothing}}
    \ge \binom{n-k-1}{k-1} \left(\gamma - \frac{k}{n-k} \sqrt{\frac{\gamma (1-\gamma)}{\beta}}\right) \quad\text{ if } \beta > 0.
    \]
\end{corollary}
\noindent (Note that we will only consider $\cat {F}$ of size at most the union of $s+1$ stars, so $\beta$ is always small and we do not need the extra factor of $1-\beta$ in the first inequality of \cref{sqrtlb}.)

\section{Proof of Theorem \ref{thm:main}}
\label{sec:lowerbound}
In this section, we prove Theorem \ref{thm:main}. Throughout, we set
\[p := \frac{k}{n} \le \frac{1}{10000s^5}.\]

\subsection{$\cat F$ is dense in some star}

 The following notation is used only in the proof of \cref{gammamax}.
\begin{notation}[Symmetric sums]
If $\vc a = (a_1, \ldots, a_n)$ and $\vc r = (r_1, \ldots, r_l)$, then we write
\[\vc a^{\vc r} := \sum_{i_1} \sum_{i_2 \neq i_1} \sum_{i_3 \notin \{i_1, i_2\}} \cdots \sum_{i_l \notin \{i_1, \ldots, i_{l-1}\}} a_{i_1}^{r_1} \cdots a_{i_l}^{r_l}.\]
For example,
\[\left(\vc a^{(1)}\right)^2 = \left(\sum_{i=1}^n a_i\right)^2 
= \sum_{i} a_i^2 + \sum_{i}\sum_{j\neq i} a_i a_j = \vc a^{(2)} + \vc a^{(1, 1)}.\]
Note that reordering the entries in $\vc r$ has no effect on the sum $\vc a^{\vc r}$.
\end{notation}
\begin{notation}[Falling factorials]
For integers $1 \leq k \leq n$ we shall write $n^{\ul{k}}$ for the $k$th \emph{falling factorial} $n(n-1)\cdots (n-k+1)$ and 
moreover for $p = \frac{k}{n}$ we write 
\[p^{\ul{i}} := \frac{k^{\ul{i}}}{n^{\ul{i}}} = \frac{k(k-1)\cdots (k-i+1)}{n(n-1)\cdots (n-i+1)}.\]
Note that $p^{\ul{i}}$ depends on $k$ and $n$ before reducing the fraction to lowest terms, and when $k \le n$, 
\[p^{\ul{i}} \le p^i.\]
\end{notation}
\begin{lemma} \label{extendsumby1}
If $\sum_{i = 1}^n a_i = 0$, then 
\[\vc a^{(r_1, r_2, \ldots, r_{l-1}, 1)} = 
-\vc a^{(r_1 + 1, r_2, \ldots, r_{l-1})}
-\vc a^{(r_1, r_2+1, \ldots, r_{l-1})}
-\cdots
-\vc a^{(r_1, r_2, \ldots, r_{l-1}+1)}.\]
\end{lemma}
\begin{proof}
Simply substitute 
\[\sum_{i_l \notin \{i_1, \ldots, i_{l-1}\}} a_{i_l} = -a_{i_1} - a_{i_2} - \cdots -a_{i_{l-1}}\]
into the innermost summand in the definition of $\vc a^{\vc r}$.
\end{proof}

In the following lemma, and its proof, we identify $\binom{[n]}{k}$ with the set
$$\left\{\vc x\in \{0, 1\}^n: \sum_{i=1}^n x_i = k\right\},$$ and we write $\EE(f)$ for the expectation of $f = f(\vc x)$ when $\vc x$ is a uniform random element of this set. So, for example, $\EE(x_1) = k/n = p$ and $\EE(x_1 x_2) = (k/n) \cdot (k-1)/(n-1) = p^{\ul{2}}$.

\begin{lemma}[Large $\norm{f_1}_2$ implies dense in some star] 
\label{lincompduetostar}
For any $n \ge 2k+1$ and any $\cat F \subseteq \binom{[n]}{k}$ with $\cat F \neq \varnothing$, let $f = \mathbbm 1_{\cat F}$ be the indicator function of $\cat F$, and $\alpha = \EE(f) = \left. \norm{f}_2^2 \middle/ \binom{[n]}{k} \right.$ be the density of $\cat F$. Let $f_1$ denote the orthgonal projection of $f$ onto $Q_1^{(k)} = \operatorname{Span}\{\vc x \mapsto x_i-k/n:\ i \in [n]\}$ (see (\ref{eq:useful})), and let $\eta = \norm{f_1}_2^2 / \norm{f}_2^2$. For each $i \in [n]$, let $\gamma_i$ denote the density of $\cat F$ in the star $\cat D_i$, i.e.
\[\gamma_i := \frac{\abs{\cat F \cap \mathcal{D}_i}}{\binom{n-1}{k-1}},\]
and let $\gamma_\text{max}$ be the largest of the $\gamma_i$'s. Then
\[\eta^3 \le \left(\frac{n-1}{n-k}\right)^3 \gamma_\text{max}^2 + 3 \left(\frac{n-1}{n-k}\right)^2 \eta \alpha.\]
\label{gammamax}
\end{lemma}
\begin{proof}

First we evaluate the coefficients of $ f_1 = \sum_{i = 1}^n a_i x_i$. 
Note that since $x_1 + x_2 + \cdots + x_n = k$ for all $\vc x$, the constant part $f_0 = \alpha$ can also be written as $ \sum_{i = 1}^n \frac{\alpha}{k} x_i$, so 
it suffices to determine the coefficients of $f_0 + f_1 = \sum_{i = 1}^n b_i x_i = \sum_{i = 1}^n (a_i + \alpha / k) x_i$.

Since $f_0+f_1$ is the orthogonal projection of $f$ onto the linear span of the functions $\vc x \mapsto x_i$, the vector $\vc b$ is the unique stationary point to the (convex) quadratic form $H = \EE\left(( \sum_{i = 1}^n b_i x_i - f)^2\right)$. The partial derivative of $H$ is given by
\begin{align*}
    \frac{\partial H}{\partial b_i} 
    &= \EE\left( 2 x_i \left( \sum_{i = 1}^n b_i x_i - f\right)\right) \\
    &= 2b_i \EE(x_i^2) + \sum_{j \neq i} 2 b_j \EE (x_i x_j) - 2\EE (x_i f) \\
    &= 2b_i p + \sum_{j \neq i} 2b_j p^{\ul{2}} - 2p \gamma_i \\
    &= 2p \left(b_i \left(1 -  \frac{k-1}{n-1}\right) + \frac{k-1}{n-1} \sum_{j = 1}^n b_j -  \gamma_i \right) \\
    &=  2p \left(b_i \frac{n - k}{n-1} + \frac{k-1}{n-1} B -  \gamma_i \right),
\end{align*}
where $B: = \sum_{j=1}^n b_j$. Hence, the stationary point satisfies 
\[b_i = \frac{n-1}{n-k} (\gamma_i + C),\]
where $C$ does not depend upon $i$.

Since $\sum_i \gamma_i = n \alpha$ (the $n$ stars cover each $k$-set equally many times, viz., $k$ times, so the average of the densities of $\cat F$ in all $n$ stars is just its density in the whole of $\binom{[n]}{k}$), and we know the $a_i$'s differ only from the $b_i$'s by an additive constant and satisfy $\sum_i a_i = 0$,
we may conclude that 
\[a_i = \frac{n-1}{n-k} (\gamma_i - \alpha) \quad \forall i \in [n].\]
We therefore have
\begin{align*} 
    \eta \alpha = \EE(f_1^2) &= \EE\left(\left(\sum_i a_i x_i\right)^2 \right)  \\
    &= \vc a^{(2)} \EE(x_1^2) + \vc a^{(1, 1)} \EE(x_1 x_2) \\
    &= \vc a^{(2)} p + \vc a^{(1, 1)} p^{\ul{2}} \\
    &= \vc a^{(2)} (p - p^{\ul{2}}) & \text{ by \cref{extendsumby1}}.
\end{align*}
By \cref{extendsumby1}, we have
\[\vc a^{(3, 1)} = -\vc a^{(4)},\]
\[ \vc a^{(2, 1, 1)} = - \vc a^{(3, 1)} -\vc a^{(2, 2)} = \vc a^{(4)} - \vc a^{(2, 2)},\]
and
\[ \vc a^{(1, 1, 1, 1)} = -3 \vc a^{(2, 1, 1)} = -3 \vc a^{(4)} +  3\vc a^{(2, 2)}. \]
Hence,
\begin{align*}
    \EE(f_1^4) &= \EE\left(\left(\sum_i a_i x_i\right)^4 \right)  \\
    &= \vc a^{(4)} \EE(x_1^4) + 4\vc a^{(3, 1)} \EE(x_1^3 x_2) + 3\vc a^{(2, 2)} \EE(x_1^2 x_2^2)+ 6 \vc a^{(2, 1, 1)} \EE(x_1^2 x_2 x_3) +\\
    & \ \quad \vc a^{(1, 1, 1, 1)} \EE(x_1 x_2 x_3 x_4) \\
    &= \vc a^{(4)} p + 4\vc a^{(3, 1)} p^{\ul{2}} + 3\vc a^{(2, 2)} p^{\ul{2}}+ 6 \vc a^{(2, 1, 1)} p^{\ul{3}} + \vc a^{(1, 1, 1, 1)} p^{\ul{4}} \\ 
    &= \vc a^{(4)} p - 4\vc a^{(4)} p^{\ul{2}} + 3\vc a^{(2, 2)} p^{\ul{2}}+ 6 (\vc a^{(4)} - \vc a^{(2, 2)}) p^{\ul{3}} + (-3 \vc a^{(4)} +  3\vc a^{(2, 2)}) p^{\ul{4}} \\
    &= \left(p - 4p^{\ul{2}} + 6p^{\ul{3}} -3p^{\ul{4}}\right)  \vc a^{(4)} + 
    \left(3p^{\ul{2}} - 6p^{\ul{3}} + 3p^{\ul{4}}\right)  \vc a^{(2, 2)} \\
    &= \left(p - 4p^{\ul{2}} + 6p^{\ul{3}} -3p^{\ul{4}}\right)  \vc a^{(4)} + 
    \left(3p^{\ul{2}} - 6p^{\ul{3}} + 3p^{\ul{4}}\right)  \left(\left(\vc a^{(2)}\right)^2 - \vc a^{(4)}\right) \\
    &= \left(p - 7p^{\ul{2}} + 12p^{\ul{3}} -6p^{\ul{4}}\right)  \vc a^{(4)} + 
    \left(3p^{\ul{2}} - 6p^{\ul{3}} + 3p^{\ul{4}}\right) \left(\vc a^{(2)}\right)^2 \\
    &\le p\vc a^{(4)} + 3p^2 \left(\vc a^{(2)}\right)^2  \hspace{20em} \text{ using } p^{\ul{2}} \ge 2 p^{\ul{3}}\\
    &\le p \left(\max_i a_i^2\right) \cdot \vc a^{(2)} + 3p^2 \left(\vc a^{(2)}\right)^2. \numberthis \label{ineq:f14}
\end{align*}
We have seen that 
\[\vc a^{(2)} = \frac{\EE(f_1^2) }{ p - p^{\ul{2}}} = \frac{1}{1 - \frac{k-1}{n-1}} \cdot \frac{ \eta \alpha }{ p} = \frac{n-1}{n-k} \cdot \frac{ \eta \alpha }{ p}.\] Also, 
\[\alpha = \frac{1}{n} \sum_{i = 1}^n \gamma_i,\]
so $\gamma_\text{max} \ge \alpha > 0$, and we have 
\[\max_i a_i^2 = \left(\frac{n-1}{n-k}\right)^2 \max_i  (\gamma_i - \alpha)^2 \le  \left(\frac{n-1}{n-k}\right)^2 \gamma_\text{max}^2,\]
so (\ref{ineq:f14}) gives 
\[ \EE(f_1^4) \le \left(\frac{n-1}{n-k}\right)^3 \gamma_\text{max}^2 \cdot \eta \alpha + 3 \left(\frac{n-1}{n-k}\right)^2 \eta^2 \alpha^2.\]
Since $f_1$ is an orthogonal projection of $f$, we have 
\begin{align*}
    \eta^4 \alpha^4 = \EE(f_1^2)^4 &= \EE(f f_1)^4 \\
    &\le \EE(f^{4/3})^3 \EE (f_1^4) &\text{ by H\" older's inequality} \\
    &\le \alpha^3 \cdot\left(\left(\frac{n-1}{n-k}\right)^3 \gamma_\text{max}^2 \cdot \eta \alpha + 3 \left(\frac{n-1}{n-k}\right)^2 \eta^2 \alpha^2 \right),
\end{align*}
so 
\[\eta^3 \le \left(\frac{n-1}{n-k}\right)^3 \gamma_\text{max}^2 + 3 \left(\frac{n-1}{n-k}\right)^2 \eta \alpha,\]
completing the proof of the lemma.
\end{proof}

\begin{lemma}[Small $\norm{f_1}_2$ implies large average degree] 
\label{thirdorder}
For any $n\ge 2 k+1$ and $\cat F \subseteq \binom{[n]}{k}$ with $\cat F \neq \varnothing$, let $f = \mathbbm 1_{\cat F}$, let $\alpha = \EE(f) = \left. \norm{f}_2^2 \middle/ \binom{[n]}{k} \right.$ and let $\eta = \norm{f_1}_2^2 / \norm{f}_2^2$. Then

\[\Delta(\cat F) \ge \frac{2e(\cat F)}{\abs{\cat F}} \ge \left(\alpha - \frac{ k}{n-k} \left(\eta + \left(\frac{k}{n-k}\right)^2\right) \right) \binom{n-k}{k}.\]
\end{lemma}
Less formally, this lemma says that if $\eta$ is small, then the average degree of the subgraph induced by $\cat F$ is not much less than $\alpha \binom{n-k}{k}$ (which is the approximate  average degree for a random family).
\begin{proof}[Proof of Lemma \ref{thirdorder}.]
In \cref{eigenvalkneser},  we have seen that the eigenvalues of the Kneser graph $K(n, k)$ are, in descending order of absolute value, as follows: 
\[\binom{n-k}{k},\ -\binom{n-k-1}{k-1},\ \binom{n-k-2}{k-2},\ \ldots,\ (-1)^k \binom{n-2k}{0}.\]
The function $f$ can be decomposed as a sum of (orthogonal) eigenvectors $f_0+f_1+\cdots +f_{k}$ corresponding to these eigenvalues, and note that $f_0 = \EE(f) \mathbbm 1 = \alpha \mathbbm 1$, so
\begin{align*}
    2e(\cat F) = \inner{f, Af} &= \binom{n-k}{k} \left\langle \alpha \mathbbm 1+ f_1 + f_2+ \cdots + f_k, \phantom{\binom{n}{k}}\right.\\ 
    &\phantom{= \binom{n-k}{k} \left\langle \alpha \right.} \alpha\mathbbm 1 - \frac{k}{n-k} f_1 + \frac{k(k-1)}{(n-k)(n-k-1)} f_2 \\
    &\phantom{\binom{n-k}{k} \left\langle \alpha \right. } \left. - \frac{k(k-1)(k-2)}{(n-k)(n-k-1)(n-k-2)} f_3 + \cdots + (-1)^k \frac{1}{\binom{n-k}{k}} f_k\right\rangle \\
    &\ge \binom{n-k}{k} \left(\alpha^2 \binom{n}{k}- \frac{k}{n-k} \alpha\eta \binom{n}{k} + \frac{k(k-1)}{(n-k)(n-k-1)} \cdot 0 \right.\\
    &\phantom{\binom{n-k}{k}} \hspace{-1em}\left. -\frac{k(k-1)(k-2)}{(n-k)(n-k-1)(n-k-2)}\left(\norm{f_3}_2^2 + \norm{f_4}_2^2 + \cdots +\norm{f_k}_2^2\right)\right) \\
    &\ge \binom{n-k}{k} \left(\alpha^2 \binom{n}{k}- \frac{k}{n-k} \alpha\eta \binom{n}{k} - \left(\frac{k}{n-k}\right)^3 \cdot  \norm{f}_2^2 \right) \\
    &= \binom{n-k}{k} \alpha\binom{n}{k} \left(\alpha - \frac{ k}{n-k} \left(\eta + \left(\frac{k}{n-k}\right)^2\right) \right) . 
\end{align*}
Dividing by $\abs{\cat F} = \alpha \binom{n}{k}$, we obtain the result.
\end{proof}

\cref{thirdorder} and \cref{gammamax} together will say if $\Delta(\cat F)$ is smaller than constructions in \cref{constructions}, then $\gamma_\text{max}$ is at least some constant. Now we shall use \cref{kneserexpander} to show that, once $\gamma_\text{max}$ is at least some constant, then indeed it cannot be much less than $\lambda/ (s+1)$.

\begin{lemma}
\label{extlem3}
If $n\ge 12sk$ and $\cat F \subseteq \binom{[n]}{k}$ has size parameter $\lambda \in [s, s+1]$, is not a union of $s$ stars, and minimizes the maximum degree subject to these conditions, then for any $i \in [n]$, if $\gamma_i = \left.\abs{\cat F_{\{i\}}^{\{i\}}}\middle/\binom{n-1}{k-1}\right. \ge c_0$, then
\[\gamma_i \ge \lambda/(s+1) - \sqrt{2(s+1)p/c_0}  - (s^2+4s) p .\]
\end{lemma}

\begin{proof}
Let $\cat B = \cat F^{\varnothing}_{\{i\}}$, $\beta = \left.\abs{\cat B}\middle/\binom{n-1}{k}\right.$, $\cat C = \cat F_{\{i\}}^{\{i\}}$,  and $
\gamma = \left.\abs{\cat C}\middle/\binom{n-1}{k-1}\right. $. By (\ref{crudesize}), we have
\[\lambda p - s^2 p^2 \le \beta (1 - p) +\gamma p \le \lambda p,\]
so 
\[\frac{(\lambda - \gamma)p - s^2 p^2 } {1-p} \le \beta \le \frac{(\lambda - \gamma)p}{1-p} < 2(s+1)p. \]
Suppose that $c_0 \le \gamma = \lambda/(s+1) - \varepsilon$. We are in exactly the setting of \cref{kneserexpander} (bipartite Kneser between $\binom{[n-1]}{k}$ and $\binom{[n-1]}{k-1}$), so 
\begin{align*}
\Delta(\cat F) \ge \frac{e(\cat B, \cat C)}{\abs{\cat C}} &\ge \binom{n - k}{k} \left( \beta - \frac{p}{1-p} \sqrt{\frac{\beta}{\gamma}}\right) \\
&\ge \binom{n - k}{k} \left(\frac{(\lambda - \gamma)p - s^2 p^2 } {1-p} - \frac{p}{1-p} \sqrt{\frac{2(s+1)p}{c_0}}\right) \\
&= \binom{n - k}{k}\frac{p}{1-p} \left((\lambda - \gamma) - s^2 p  - \sqrt{\frac{2(s+1)p}{c_0}}\right)\\
&= \binom{n - k -1}{k-1} \left(\left(\frac{s\lambda}{s+1} +\varepsilon\right) - s^2 p  - \sqrt{\frac{2(s+1)p}{c_0}}\right).
\end{align*}
Comparing this with the upper bound (\ref{upperbound}), we may conclude that 
\[\varepsilon \le (s^2 +4s) p + \sqrt{2(s+1)p/c_0}. \qedhere\] 
\end{proof}

\subsection{$\cat F$ is dense in $s+1$ stars}

The aim of this subsection is to show that if $\Delta(\cat F)$ is small, then \cref{extlem3} can be used several times to give $s+1$ `popular' elements in the family $\cat F$. Before proving the main result \cref{manylem3}, we need some straightforward bounds on binomial coefficients.

\begin{lemma}\label{binomratio}
Let $k \le m \le n$, then 
\[\left.\binom{m}{k} \middle/ \binom{n}{k}\right. \ge 1 - \frac{k(n-m)}{n-k+1}.\]
\end{lemma}
\begin{proof}
We have
\begin{align*}\left.\binom{m}{k} \middle/ \binom{n}{k}\right. 
&= \frac{m(m-1)\cdots (m-k+1)}{n(n-1)\cdots(n-k+1)}  \\
&\ge \left(\frac{m-k+1}{n-k+1}\right)^k\\
&= \left(1 - \frac{n-m}{n-k+1}\right)^k \\
&\ge 1 - \frac{k(n-m)}{n-k+1},
\end{align*}
as required.
\end{proof}

\begin{lemma}[Conversion between $\binom{[n]}{k}$ and $\binom{[n]\setminus [s]}{k}$]
\label{convert}
Suppose that $s \leq  k \leq n/2$ and that $\cat F \subseteq \binom{[n]}{k}$. Writing \[\gamma := \left.\abs{\cat F_{\{1\}}^{\{1\}}} \middle/ \binom{n - 1}{k - 1}\right.
\quad \text{and } \quad \widetilde \gamma := \left.\abs{\cat F^{\{1\}}_{[s+1]} }\middle/ \binom{n - s -1}{k-1}\right. ,\]
we have
\[\widetilde \gamma + sp> \gamma \ge  \widetilde \gamma (1 - 2sp). \]
\end{lemma}
\begin{proof}
Since $\abs{\cat F_{\{1\}}^{\{1\}}} \ge \abs{\cat F^{\{1\}}_{[s+1]} }$, we have $\gamma \binom{n-1}{k-1} \ge \widetilde \gamma \binom{n-s-1}{k-1}$, so
\begin{align*}
    \gamma &\ge \widetilde \gamma \left.\binom{n-s-1}{k-1} \middle/ \binom{n-1}{k-1}\right.   \\
    &\ge  \widetilde \gamma \left(1 - \frac{(k-1)s}{n-k+1}\right) &\text{by \cref{binomratio}} \\
    &\ge  \widetilde \gamma (1 - 2sp).
\end{align*}
On the other hand,
\[\abs{\cat F^{\{1\}}_{[s+1]} } \ge \abs{\cat F_{\{1\}}^{\{1\}}} - \sum_{i = 2}^{s+1} \abs{\cat F_{\{1, i\}}^{\{1, i\}}},\]
so\begin{align*} 
\widetilde \gamma \binom{n-s-1}{k-1} &\ge \gamma \binom{n-1}{k-1} - s\binom{n-2}{k-2}  \\
&= \left(\gamma - \frac{s(k-1)}{n-1}\right)\binom{n-1}{k-1} \\
&> (\gamma - sp) \binom{n-1}{k-1},\end{align*}
so $\widetilde \gamma > \gamma - sp$ (this follows from the above if $\gamma - sp \ge 0$, and is trivially true if $\gamma -sp < 0$).
\end{proof}

\begin{lemma}
\label{manylem3}
If $n \ge 100 s^2 k$ and $\cat F \subseteq \binom{[n]}{k}$ has size parameter $\lambda \in [s, s+1]$, is not a union of $s$ stars, and minimizes the maximum degree subject to these conditions, then there are at least $s+1$ elements $x \in [n]$ such that 
\[\abs{\cat F_{\{x\}}^{\{x\}}} \ge \binom{n-1}{k-1}\left(\frac{\lambda}{s+1} - \sqrt{40(s+1)p}  - (s^2+4s) p \right).\]
\end{lemma}

\begin{proof}
Pick $s$ arbitrary elements from $[n]$, say $1, 2, \ldots, s$; we shall show there exists some other element $x \in [n] \setminus [s]$ such that $\abs{\cat F_{\{x\}}^{\{x\}}}$ satisfies the above bound. (Note that no assumptions here are made on the sizes of $\cat F_{\{1\}}^{\{1\}}, \cat F_{\{2\}}^{\{2\}}, \ldots, \cat F_{\{s\}}^{\{s\}}$.) This in turn implies that at least $s+1$ elements of $[n]$ satisfy the inequality.

We consider the partition of $\cat F$ into $2^s$ parts according to how each $A \in \cat F$ intersects $[s]$. Each of these $2^s$ parts corresponds to some $\cat F_{[s]}^I$ (for $I \subseteq [s]$).

Formally, let $\cat B := \cat F_{[s]}^\varnothing$, $\beta:= \abs{\cat B}/\binom{n-s}{k}$, and for $I \neq \varnothing$ let 
$\cat C_I: = \cat F_{[s]}^I$, and $\delta_I  := \abs{\cat C_I }/\binom{n-s }{k-\abs{I}}$. Equation (\ref{Fsize}) now gives
\[\beta \binom{n - s}{k} + \sum_{\substack{I \subseteq [s] \\ I \neq \varnothing}} \delta_I 
\binom{n - s}{k - \abs{I}} = (\lambda - s) \binom{n - s - 1}{k-1} + \sum_{j=1}^{s}\binom{s}{j}\binom{n-s}{k-j}.\]
Using $\delta_I \le 1$ for all $I$ with $\abs{I} \ge 2$, we have 
\[\beta \binom{n - s}{k} + \sum_{i = 1}^s \delta_i 
\binom{n - s}{k - 1} \ge (\lambda - s) \binom{n - s - 1}{k-1} + \binom{s}{1}\binom{n-s}{k-1}\]
where $\delta_i = \delta_{\{i\}}$. Dividing by $\binom{n-s}{k-1}$, we obtain
\begin{equation}\beta \frac{n - s - k + 1}{k} + \sum_{i = 1}^s \delta_i 
\ge (\lambda - s) \frac{n - s - k + 1}{n - s} + s \ge s \quad\text{since } \lambda \ge s, \label{betadeltalb}
\end{equation}
and 
\begin{equation*}
    \beta \ge \frac{k}{n - s - k + 1} \sum_{i = 1}^s ( 1 - \delta_i).
\end{equation*}
In particular, we have
\begin{equation}
    \beta \ge \frac{k}{n - s - k + 1}  ( 1 - \delta_i) \quad \text{for each } i. \label{betalb}
\end{equation}
By assumption, $\cat F$ is not a union of $s$ stars, so $\beta >0$. Note that $\cat B \subseteq \binom{[n]\setminus [s]}{k}$ and $\cat C_{\{i\}} \subseteq \binom{[n]\setminus [s]}{k-1}$, so we are in the setting of \cref{kneserexpander} (except that $n-1$ is replaced by $n - s$), so
\begin{align} \frac{e(\cat B, \cat C_{\{i\}})}{\abs{\cat B}} &\ge \binom{n - s - k}{k - 1} 
\left(\delta_i - \frac{k}{n - s - k + 1} \sqrt{\frac{\delta_i (1-\delta_i)}{\beta}}\right) 
\nonumber \\
&\ge \binom{n - s - k}{k - 1} 
\left(\delta_i - \sqrt{\frac{k}{n - s - k + 1}} \right) 
& \text{using (\ref{betalb}) and } \delta_i \le 1. \nonumber \\
&\ge \binom{n - s - k}{k - 1} 
\left(\delta_i - \sqrt{2p} \right). \label{BtoCi}
\end{align}
(The last line uses $p \le 1/(4s)$, so that $n \ge 4sk  \ge 2(s + k)$ and $\frac{k}{n-s-k} \le \frac{2k}{n} = 2p$.)

Applying \cref{thirdorder} to $\cat B \subseteq \binom{[n] \setminus [s]}{k}$, we have
\begin{align}\frac{2 e (\cat B) }{\abs{\cat B}} &\ge \binom{n - s - k}{k} \left(\beta - \frac{ k }{n - s - k} \left(\zeta +  \left(\frac{ k }{n - s - k}\right)^2 \right)\right)\nonumber\\
&=  \binom{n - s - k}{k - 1} \frac{n - s - 2k + 1}{k} \left(\beta - \frac{ k }{n - s - k} \left(\zeta +  \left(\frac{ k }{n - s - k}\right)^2 \right)\right)\nonumber \\ 
&\ge \binom{n - s - k}{k - 1} \left( \frac{n - s - 2k + 1}{k} \beta - \frac{ n-s - 2k + 1 }{n - s - k} \left(\zeta +  (2p)^2 \right)\right)\nonumber \\ 
&\ge \binom{n - s - k}{k - 1} \left( \frac{n - s - 2k + 1}{k} \beta - \zeta -4p^2\right),\label{BtoB}
\end{align}
where $\zeta = \norm{g_1}_2^2 / \norm{g}_2^2$ for $g = \mathbbm 1_{\cat B}: \binom{[n]\setminus [s]}{k} \to \{0, 1\}$, and $g_1$ denotes the `linear component' of $g$, i.e.\ the orthogonal projection of $g$ onto $Q_1^{(k)} = \text{Span}\{\mathbf{x} \mapsto x_i-k/(n-s):\ i \in [n]\setminus [s]\}$.

By (\ref{BtoB})$+ \sum_{i = 1}^s$(\ref{BtoCi}) (the number of edges can be simply added together, because the $\cat C_{\{i\}}$'s correspond to disjoint subfamilies of $\cat F$), we have
\begin{align}
    \phantom{\ge .}\Delta(\cat F) \nonumber
    &\ge \frac{\sum_{i = 1}^s e(\cat B, \cat C_{\{i\}}) + 2e(\cat B) }{\abs{\cat B}} \nonumber\\
    &\ge \binom{n - s - k}{k - 1} \left(\sum_{i = 1}^s \delta_i  - s\sqrt{2p}+ \frac{n - s - 2k + 1}{k} \beta - \zeta -4p^2\right)  \nonumber\\
     &\ge \binom{n - s - k}{k - 1} \left(\frac{n - s - 2k + 1}{n - s - k + 1} \left( \frac{n - s - k + 1}{k}\beta + \sum_{i = 1}^s \delta_i\right)- s\sqrt{2p}
      - \zeta -4p^2\right) \nonumber\\
     &\ge \binom{n - s - k}{k - 1} \left(\frac{n - s - 2k + 1}{n - s - k + 1} \left( (\lambda - s) \frac{n - s - k + 1}{n - s} + s\right)- s\sqrt{2p} - \zeta - 4p^2\right),  \hspace{2em}\label{s+1lowerbound}
\end{align}
where we have used (\ref{betadeltalb}) for the last inequality. We will now to combine this lower bound (\ref{s+1lowerbound}) with the upper bound (\ref{upperbound}). By \cref{binomratio}, we have
\begin{align}
    \binom{n - s - k}{k - 1} &\ge \binom{n - k - 1}{k-1} \left(1-\frac{(k-1)(s-1)}{n-2k+1}\right)
    \nonumber\\
    &\ge \binom{n - k - 1}{k-1}(1-2sp). \label{convertbinom}
\end{align}
Combining (\ref{s+1lowerbound}) and (\ref{upperbound}), and dividing through by $\binom{n-k-1}{k-1}$, we obtain 
\begin{align*}
&\phantom{\ge 1} \frac{s \lambda}{s+1} + 4sp \\
&\ge   (1-2sp) \left(\frac{n - s - 2k + 1}{n - s - k + 1} \left( (\lambda - s) \frac{n - s - k + 1}{n - s} + s\right)- s\sqrt{2p} - \zeta - 4p^2\right)\\
&\ge (1 - 2sp) \left((1-2p) \left((\lambda - s) (1-2p) + s\right) - s \sqrt{2p} - \zeta - 4p^2\right)\\
&\ge (1-(2s+4)p)(\lambda - s) + (1-(2s+2)p) s- s \sqrt{2p}  - \zeta - 4p^2.
\end{align*}
This inequality can be rewritten as 
\begin{align*}&\phantom{\ge 1}\left(1 - \frac{1}{s+1}\right) (\lambda - s) + \left(1 - \frac{1}{s+1}\right) s + 4sp \\ 
&\ge (1-(2s+4)p)(\lambda - s) + (1-(2s+2)p) s- s \sqrt{2p}  - \zeta - 4p^2,
\end{align*}
so
\begin{align*}\zeta &\ge \left( \frac{1}{s+1} - (2s+4)p \right) (\lambda - s) + \left(\frac{1}{s+1} - (2s+2)p\right) s -s\sqrt{2p} - 4p^2 -  4sp.
\end{align*}
When $p \le 1/(100s^2),$
the above inequality implies that
\[\zeta \ge 0 + \frac{1}{1.5 (s+1)} \cdot s  - s\sqrt{2p} - 4p^2 - 4sp > \frac{s}{2(s+1)} \ge \frac{1}{4}.\]
Now apply \cref{gammamax} to the family $\cat B \subseteq \binom{[n]\setminus [s]}{k}$. This lemma says there is some $x \in [n] \setminus [s]$ such that $\delta':=\left.\abs{\cat B_{\{x\}}^{\{x\}}} \middle/ \binom{n-s-1}{k-1}\right.$ satisfies
\[\frac{1}{64} \le \zeta^3 \le \left(\frac{n-s-1}{n-s-k}\right)^3 \delta'^2 + 3 \left(\frac{n-s-1}{n-s-k}\right)^2 \zeta \beta.\]
Using $\beta \le 2sp$,  $k/(n-s) \le 2p$, and $p \le 1/(100s^2)$, 
    we may conclude that the density of $\cat B_{\{x\}}^{\{x\}} = \cat F_{[s] \cup \{x\}}^{\{x\}}$ in $\binom{[n] \setminus ([s] \cup \{x\})}{k-1}$ is $\delta' \ge 1/10$, so by \cref{convert}, $\cat F_{\{x\}}^{\{x\}}$ has density $\delta \ge \delta' (1-2sp) > 1/20$. By \cref{extlem3}, we have
\[\abs{\cat F_{\{x\}}^{\{x\}}} \ge \binom{n-1}{k-1}\left(\frac{\lambda}{s+1} - \sqrt{40(s+1)p}  - (s^2 +4s)p \right),\]
as required. 
\end{proof}
\begin{remark}
    If $n$ is sufficiently large, then it is not possible to have $s+2$ such elements because $\lambda (s+2) / (s+1) > \lambda$ and by inclusion--exclusion $\cat F$ would be too large.
\end{remark}
\subsection{Calculating the maximum degree}
Now we have a structural description that most of $\cat F$ can be covered by $s+1$ stars. This allows us to calculate the maximum degree.

\begin{theorem}
\label{lowerbound}
If $n \ge 10000s^2 k$, $\cat F \subseteq \binom{[n]}{k}$ has size parameter $\lambda \in [s, s+1]$, and $\cat F$ is not a union of $s$ stars, then 
\[\Delta(\cat F) \ge \binom{n - k - 1}{k-1}  \left(\frac{s\lambda}{s+1} - 11\left(\sqrt{s^3 p} + s^3 p\right)\right).\]
\end{theorem}
\begin{proof} 
Let $\cat F$ be a family that minimizes the maximum degree, subject to the conditions of the theorem. We may assume that $1, 2, \ldots, s+1$ are the $s+1$ elements given by \cref{manylem3}. So for $i = 1, \ldots, s+1$, we have
\[
\gamma_i = \left.\abs{\cat F_{\{i\}}^{\{i\}}}\middle/\binom{n-1}{k-1} \right.\ge \frac{\lambda}{s+1} - \sqrt{40(s+1)p}  - (s^2+4s) p .
\]
Now let $\cat B = \cat F_{[s+1]}^{\{s+1\}}$ and $\cat C_{\{i\}} = \cat F^{\{i\}}_{[s+1]}$. (Note that this contrasts with the proof of \cref{manylem3}, where we defined $\cat C_{\{i\}}$ to be $\cat F^{\{i\}}_{[s]}$.) \cref{convert} gives
\[\widetilde\gamma_i := \left.\abs{\cat F_{[s+1]}^{\{i\}}}\middle/\binom{n-s-1}{k-1} \right. > \gamma_i - sp.\]
So, when viewed as a subset of $\binom{[n] \setminus [s]}{k-1}$,
the density of $\cat C_{\{i\}}$ is 
\begin{equation}\delta_i := \left.\abs{\cat F_{[s+1]}^{\{i\}}}\middle/\binom{n-s}{k-1}\right. 
= \frac{n-s-k+1}{n-s} \widetilde\gamma_i > (1-2p) (\gamma_i - sp) >\gamma_i - (s+2)p. \label{ineq:delta}\end{equation}
Also, we have
\[\abs{\cat B} > (\gamma_{s+1} - sp)\binom{n - s - 1}{k - 1} = (\gamma_{s + 1} - sp)\frac{k}{n-s} \binom{n - s}{k}.\]
Using $\frac{\lambda}{s+1} \ge \frac{s}{s+1} \ge \frac{1}{2}$, the density of $\cat B$ in $\binom{[n]\setminus [s]}{k}$ is 
\begin{equation}\beta > (\gamma_{s + 1} - sp)\frac{k}{n-s} \ge \frac{1}{3} \cdot \frac{k}{n-s} \ge \frac{k}{n - k - s + 1} \cdot \frac{1}{4}. \label{betalb2}\end{equation}
(Technically, $\cat B \subseteq \binom{[n]\setminus [s+1]}{k-1}$, but the latter injects into $\binom{[n]\setminus [s]}{k}$ by adding the element $s+1$ back into every set.) Now we apply (\ref{BtoCi}) to 
$\cat B \subseteq \binom{[n]\setminus [s+1]}{k-1} \hookrightarrow \binom{[n]\setminus [s]}{k}$ (noting that $\cat B$ has density $\beta$ in the latter set) and $\cat C_{\{i\}} \subseteq \binom{[n]\setminus [s]}{k-1}$ (noting that $\cat C_{\{i\}}$ has density $\delta_i$ in the latter set), obtaining:
\begin{align*}
    &\phantom{\ge .}\Delta(\cat F) \\
    &\ge \frac{\sum_{i = 1}^s e(\cat B, \cat C_{\{i\}})  }{\abs{\cat B}} \\
    &\ge  \binom{n - s - k}{k - 1} 
\sum_{i = 1}^s \left(\delta_i - \frac{k}{n - s - k + 1} \sqrt{\frac{\delta_i (1-\delta_i)}{\beta}}\right) 
 \\
&\ge \binom{n - s - k}{k - 1} 
\sum_{i = 1}^s \left(\delta_i - \sqrt{\frac{k}{n - s - k + 1}} \right) \hspace{5em} \text{using }\delta_i (1-\delta_i) \le 1/4 \text{ and (\ref{betalb2})}\\
&\ge \binom{n - s - k}{k - 1} \sum_{i = 1}^s \left(  (\gamma_i - (s+2)p) - \sqrt{2p}\right) \hspace{4.5em} \text{ by (\ref{ineq:delta})}\\
&\ge \binom{n - k - 1}{k - 1} (1-2sp) \left(\frac{s\lambda}{s+1} - s\sqrt{40(s+1)p}  - (s^3+4s^2) p  - s(s+2) p - s \sqrt{2p}\right) \\
& \hspace{24em} \text{ by (\ref{convertbinom})} \\
& \ge \binom{n - k - 1}{k-1}  \left(\frac{s\lambda}{s+1} - s\sqrt{40(s+1)p}  - (s^3+9s^2) p- s \sqrt{2p}\right)
\\
& \hspace{2em} \text{since the last factor on previous line is } \leq s, \text{ and if } x \le s \text{ then } (1-2sp) x \le x - 2s^2 p\\
&= \binom{n - k - 1}{k-1}  \left(\frac{s\lambda}{s+1} - 11\left(\sqrt{s^3 p} + s^3 p\right) \right).  
\end{align*}

\end{proof}

\begin{remark}
    
    If we want to see the ``jump'' behaviour as in \cref{roughplot} (with a jump of $\varepsilon$, say) from \cref{lowerbound}, then we need
    
    \[s^2 / (s+1) - 11\left(\sqrt{s^3 p} + s^3 p\right) > s-1 + \varepsilon > s - 1.\]
    (The left hand side corresponds to the lower bound from \cref{lowerbound} when $\lambda > s$, and the right-hand side corresponds to the case where $\cat F$ is a union of $s$ stars.) Hence, the error $O\left(\sqrt{s^3 p} + s^3 p\right)$ should not exceed $c/(s+1)$ for some small constant $c < 1$, so $p \leq c's^{-5}$ (for some small absolute constant $c'>0$) is necessary.
\end{remark}
\begin{remark}
    Note that the $O(s^3 p)$ part of the error term in the statement of the previous theorem resulted from our use of inclusion-exclusion, whereas the dominating $O\left(\sqrt{s^3 p}\right)$ part of the error term came from our use of the Expander Mixing Lemma.
\end{remark}
We obtain the following.
\begin{corollary}
\label{equalitycase}
If $n \ge 10000 s^5 k$, $\cat F \subset {[n] \choose k}$ has size parameter $\lambda = s$, and $\cat F$ minimizes the maximum degree over all induced subgraphs of $K(n,k)$ of the same order, then $\cat F$ is a union of $s$ stars.
\end{corollary}
\begin{proof}
    A union of $s$ stars has maximum degree:
    \begin{align} 
    \label{eq:starsmax}
    &\phantom{= } \binom{n-k}{k} - \binom{n-k-s+1}{k} \nonumber \\ 
    &= (s-1) \binom{n-k-s+1}{k-1} + \sum_{i=2}^{s-1}\binom{s-1}{i}\binom{n-k-s+1}{k-i}.\end{align}
If $\cat F$ is not a union of $s$ stars, then \cref{lowerbound} says it has maximum degree at least
    \begin{align*}&\phantom{> }\binom{n - k - 1}{k-1}  \left(\frac{s^2}{s+1} -  11\left(\sqrt{s^3 p} + s^3 p\right)\right) \\
    &>  \binom{n-k-1}{k-1} (s-1)\\
    &= (s-1) \sum_{i=1}^{s-1}\binom{s-2}{i-1} \binom{n-k-s+1}{k-i},\end{align*}
    which is at least as large as (\ref{eq:starsmax}), since
    $$ (s-1){s-2 \choose i-1} \geq {s-1 \choose i}\quad \forall i \geq 1.$$
\end{proof}

\section{Relation to the Erd\H{o}s matching conjecture}
\label{reltoEMC}
Recall that a {\em matching} is a family of pairwise disjoint sets. Recall the following well-known problem, posed originally by Erd\H{o}s.
\begin{question}[Erd\H{o}s matching problem, \cite{EMCorig}]
    \label{EMC}
    Consider all $\cat F \subseteq \binom{[n]}{k}$ of size $\abs{\cat F} = m$. What is the minimum value of $\omega(K(n,k)[\cat F])$? That is, for each triple $(n,k,m) \in \mathbb{N}^3$, what is the minimum possible size of the largest matching in $\mathcal{F}$, over all $m$-element subsets of ${[n] \choose k}$?
    \end{question}

    For $s \in \mathbb{N}$, there are two natural constructions of large families $\cat F$ with $\omega(K(n,k)[\cat F]) \le s$.
    One family $\cat A$ is a union of $s$ stars:
    $$\mathcal{A} = \cat D_1 \cup \cat D_2 \cup \cdots \cup \cat D_s = \left\{A \in {[n] \choose k}:\ A \cap [s] \neq \varnothing\right\}.$$
    Among any $s+1$ sets in $\cat A$, two must be from the same one of the $s$ stars, so $s+1$ sets in $\mathcal{A}$ cannot be pairwise disjoint. This family has
    \[\abs{\cat A} =  \binom{n}{k} - \binom{n-s}{k}.\]
     Another construction is to take $\mathcal{B} = {[k(s+1)-1] \choose k}$, i.e.\ the family of all $k$-element subsets of $[k(s+1)-1]$; this has
    \[\abs{\cat B} = \binom{k(s+1)-1}{k},\]
    and clearly it is not possible for $(s+1)$ pairwise disjoint subsets of size $k$ to fit inside $[k(s+1)-1]$.
    The celebrated {\em Erd\H{o}s Matching Conjecture} says that, if $n \ge k(s+1)$, and $\cat F \subseteq \binom{[n]}{k}$ has size $\abs{\cat F} > \max \{\abs{\cat A}, \abs{\cat B}\}$, then 
    \[\omega(K(n, k)[\cat F]) \ge s+1.\]
This conjecture remains open, though several partial results are known. Erd\H{o}s himself proved the conjecture
for all $n$ sufficiently large, i.e.\ for $n \geq n_0(k,s)$. The bound on $n_0(k,s)$ was lowered in several works: Bollob\'as, Daykin and Erd\H{o}s~\cite{BDE76} showed
that $n_0(k, s) \leq 2sk^3$; Huang, Loh and Sudakov~\cite{HLS12} showed that $n_0(k,s) \leq 3sk^2$, and Frankl and F\"{u}redi (unpublished)
showed that $n_0(k, s) \leq cks^2$. A breakthrough was achieved by Frankl~\cite{Frankl13} in 2013; Frankl showed that $n_0(k,s) \leq (2s+1)k-s+1$. This was then further sharpened by Frankl and Kupavskii \cite{erdosmatchingconj}; they proved the following.
\begin{theorem}[Frankl--Kupavskii, Theorem 1 in \cite{erdosmatchingconj}]
There exists $s_0 \in \mathbb{N}$ such that for all $s \in \mathbb{N}$ with $s \geq s_0$, all $n \in \mathbb{N}$ with $n \geq \frac{5}{3} sk - \frac{2}{3} s$, and all $\cat F \subseteq \binom{[n]}{k}$,
\[\abs{\cat F} > \abs{\cat A} \implies 
    \omega(K(n, k)[\cat F]) \ge s+1.\]
\end{theorem}

A good lower bound for the $\lambda > 2$ case in the maximum degree problem implies the corresponding instance of the Erd\H{o}s Matching Conjecture with $s = 2$. Similarly, the result for larger $\lambda$ implies the Erd\H{o}s Matching Conjecture for a corresponding range of parameters, as we proceed to outline in the following.

\begin{corollary}
    If the triple $(n, k, s) \in \mathbb{N}^3$ satisfies $n \ge 10000s^5 k$, then the Erd\H{o}s Matching Conjecture holds for this triple --- that is, for any $\cat F\subseteq \binom{[n]}{k}$ with size greater than the union of $s$ stars, $\cat F$ has a matching of size $s+1$, i.e.\ $K(n,k)[\mathcal{F}]$ contains a clique of size $s+1$.
\end{corollary}
\begin{proof} 
    Let $(n, k, s) \in \mathbb{N}^3$ be such that $n \geq 10000s^5k$, and let $\cat F \subset {[n] \choose k}$ such that $|\mathcal{F}| > {n \choose k}-{n-s \choose k}$. We need to show $\cat F$ has a matching of size $s+1$. We already know that the maximum degree of $K(n,k)[\mathcal{F}]$ is large, by \cref{lowerbound}. We shall pick out a vertex of maximum degree and focus on its neighbourhood, and then repeat this process. If we can do this $s$ times, and the remaining neighbourhood (adjacent to all $s$ chosen vertices) is still non-empty, then we have an $(s+1)$-clique in $K(n,k)[\mathcal{F}]$, and we are done.

    Since $\cat F$ has size parameter $\lambda > s$, \cref{lowerbound} yields
    \[\Delta (K(n,k)[\cat F])  \ge \binom{n - k - 1}{k-1}  \left(\frac{s^2}{s+1} -  11\left(\sqrt{s^3 p} + s^3 p\right)\right),\]
where, as before, $p=k/n$.

Let $v_1 \in \mathcal{F}$ be a vertex of degree $\Delta(K(n,k)[\cat F])$. Then the neighbourhood $\Gamma (v_1)$ is a subfamily of $ \binom{[n]\setminus v_1}{k}$, the latter isomorphic to $\binom{[n-k]}{k}$. Moreover the subset $\Gamma (v_1)$ has size $\Delta (K(n,k)[\cat F])$, which is greater than $(s-1)\binom{n-k-1}{k-1}$, since
$$\frac{s^2}{s+1} -  11\left(\sqrt{s^3 \tfrac{k}{n}} + s^3 \tfrac{k}{n}\right) > s-1,$$
as is easy to check. Hence, $\Gamma (v_1)$ has size greater than the union of $s-1$ stars in $\binom{[n-k]}{k}$. 

    It is clear that if $(n, k, s)$ satisfies the condition $n \geq 10000s^5k$, then $(n-k, k, s-1)$ also satisfies the same condition, so we can apply the same argument to the family $\cat F_1 := \Gamma(v_1) \subseteq \binom{[n]\setminus v_1}{k} \cong \binom{[n-k]}{k}$. Similarly, we can repeat this until we obtain
    $\cat F_s \subseteq \binom{[n]\setminus (v_1 \cup v_2 \cup \cdots\cup v_s)}{k} \cong \binom{[n-sk]}{k}$, which has size at least
    $$\binom{n - k - 1}{k-1}  \left(\frac{1}{2} -  11\left(\sqrt{\tfrac{k}{n}} + \tfrac{k}{n}\right)\right)>0,$$
    so we are done.
\end{proof}

This is of course far worse than the result in \cite{erdosmatchingconj}, which only requires $ n \ge 5sk/3 - 2s/3 $; we include it only to illustrate the connection between the two problems.

\end{document}